\documentclass[11pt]{article}
\usepackage{color,amsmath, amsthm, amsfonts, amssymb, dsfont, graphicx, listings, graphics, epsfig, enumerate, bbm,dsfont,tikz,mathrsfs,subcaption
}

\usepackage{pdflscape,adjustbox,array,caption}

\usepackage{diagbox}
\usepackage[margin=1in]{geometry}
\usepackage[fleqn,tbtags]{mathtools}
\usepackage[english]{babel}
\usepackage{setspace}
\usepackage[round]{natbib} 
\usepackage{algorithm,algpseudocode}
\usepackage{comment}
\usepackage{adjustbox}
\usepackage{colortbl}
\usepackage{color}
\definecolor{lightblue}{rgb}{0.57,0.85,0.96}
\definecolor{lightgray}{rgb}{0.75,0.75,0.75}
\definecolor{lightred}{rgb}{1,0.6,0.6}
\definecolor{lightgreen}{rgb}{0.77,1,0.28}
\definecolor{lightorange}{rgb}{1,0.74,0.1}
\definecolor{lightpurple}{rgb}{0.92,0.67,0.92}
\definecolor{lightbrown}{rgb}{1,0.74,0.4}
\newcolumntype{L}{>{\arraybackslash}p{0.8cm}}
\newcolumntype{P}[1]{>{\centering\arraybackslash}p{#1}}
\title{Constructive covers of a finite set}
\author{\c{C}a\u{g}{\i}n Ararat\thanks{Bilkent University, Department of Industrial Engineering, Ankara, Turkey, cararat@bilkent.edu.tr.}\and  \"{U}lk\"{u} G\"{u}rler\thanks{Bilkent University, Department of Industrial Engineering and Faculty of Business Administration, Ankara, Turkey, ulku@bilkent.edu.tr.}\and M.~Emrullah Ild{\i}z\thanks{Middle East Technical University, Department of Electrical and Electronics Engineering, Ankara, Turkey, emrullah.ildiz@metu.edu.tr.}
}
\date{July 2, 2020}

\addto\captionsenglish {}
\makeatletter \renewenvironment{proof}[1][\proofname] {\par\pushQED{\qed}\normalfont\topsep6\p@\@plus6\p@\relax\trivlist\item[\hskip\labelsep\bfseries#1\@addpunct{.}]\ignorespaces}{\popQED\endtrivlist\@endpefalse} \makeatother
\newtheorem{theorem}{Theorem}[section]
\newtheorem{corollary}[theorem]{Corollary}

\newtheorem{proposition}[theorem]{Proposition}

\newtheorem{definition}[theorem]{Definition}
\theoremstyle{definition}
\newtheorem{example}[theorem]{Example}
\newtheorem{remark}[theorem]{Remark}

\numberwithin{equation}{section}

\newcommand{\sm}{\!\setminus\!}

\let\abs=\envert

\newcommand{\D}{\mathcal{D}}
\newcommand{\W}{\mathcal{W}}

\newcommand{\Y}{\mathcal{Y}}
\newcommand{\Z}{\mathbb{Z}}
\newcommand{\A}{\mathscr{A}}

\renewcommand{\a}{\alpha}
\renewcommand{\b}{\beta}

\renewcommand{\S}{\mathbb{S}}
\newcommand{\N}{\mathbb{N}}
\newcommand{\NN}{\mathcal{N}}

\newcommand{\T}{\mathbb{T}}

\newcommand{\of}[1]{\ensuremath{\left( #1 \right)}}
\newcommand{\cb}[1]{\ensuremath{ \left\{ #1 \right\} }}

\def\prehp(#1,#2){\ensuremath{  #1 \cdot #2 }}

\begin{document}
\maketitle
\thispagestyle{empty}
\begin{abstract}
Given positive integers $n,k$ with $k\leq n$, we consider the number of ways of choosing $k$ subsets of $\cb{1,\ldots,n}$ in such a way that the union of these subsets gives $\cb{1,\ldots,n}$ and they are not subsets of each other. We refer to such choices of sets as \emph{constructive $k$-covers} and provide a semi-analytic summation formula to calculate the exact number of constructive $k$-covers of $\cb{1,\ldots,n}$. Each term in the summation is the product of a new variant of Stirling numbers of the second kind, referred to as integrated Stirling numbers, and the cardinality of a certain set which we calculate by an optimization-based procedure with no-good cuts for binary variables.\\
\\[-5pt]
\\[-5pt]
\end{abstract}

\section{Introduction}\label{intro}

Let $\NN$ be a set with $n\in\N\coloneqq\cb{1,2,\ldots}$ distinct objects and fix $k\in\N$. We consider the problem of choosing $k$ subsets $A_1,\ldots,A_k$ of $\NN$ in such a way that these sets cover $\NN$, that is, $\bigcup_{i=1}^k A_i = \NN$, and no two of these sets are subsets of each other, that is, $A_i \setminus A_j \neq \emptyset$ for all $i,j\in\cb{1,\ldots,k}$ with $i\neq j$. We refer to such a cover of $\NN$ as a \emph{constructive $k$-cover}. Due to the second requirement in the preceding definition, it is necessary that $k\leq n$ for a constructive $k$-cover to exist. In this paper, we address the following question: \emph{How many ways are there to choose a constructive $k$-cover of $\NN$?}

In the literature, a related problem was considered by T.~Hearne, C.~Wagner in 1973 and by R.~J.~Clarke in 1990. In both of \cite{hearnewagner} and \cite{clarke}, distinct subsets $A_1,\ldots,A_k$ of $\NN$ are said to form a \emph{minimal $k$-cover} of $\NN$ if these sets cover $\NN$, and $\NN$ cannot be a covered by a strict subcollection of these sets, that is, $\bigcup_{i\neq j}A_i\neq\NN$ for all $j\in\cb{1,\ldots,k}$. It is easy to observe that a minimal $k$-cover is necessarily a constructive $k$-cover. As a counterexample for the converse, when $n=4$ and $k=3$, note that the sets $A_1=\cb{1,2}, A_2=\cb{2,3}, A_3=\cb{1,3,4}$ form a constructive $k$-cover but not a minimal $k$-cover. In \cite{hearnewagner} and \cite{clarke}, summation-type formulae are provided for the number of minimal $k$-covers of $\NN$ (in \cite{clarke}, also for the case where the elements of $\NN$ are not necessarily distinguishable). In the present paper, the requirement that no two sets be subsets of each other makes the problem more involved and it seems that the enumeration ideas in \cite{clarke} cannot be replicated here to solve the problem. The one-to-one correspondence between minimal covers and the so-called \emph{split graphs} on $n$ vertices is studied in \cite{royle}. The more recent work \cite{collinstrenk} provides a more detailed analysis of the connection between minimal covers, split graphs and bipartite posets. In a similar line of research, the so-called \emph{$m$-balanced covers} which consist of sets with cardinality $m$ are studied in \cite{burger} and a recursive relation is provided for the number of $m$-balanced covers that include the minimum possible number of sets.

The consideration of constructive $k$-covers is motivated by \emph{coherent systems} in reliability theory. In \cite{barlow}, a coherent system is described in terms of its set $\NN$ of distinct components and the way these components are configured. The precise configuration of the components is encoded by the so-called \emph{minimal path sets} $P_1,\ldots,P_k\subseteq\NN$, which are supposed to form a constructive $k$-cover of $\NN$. For each $i\in\cb{1,\ldots,k}$, the components in $P_i$ are interpreted as components that are connected in series, and the sub-systems $P_1,\ldots,P_k$ are connected in parallel. The class of coherent systems is rich enough to cover many system configurations in reliability applications. For instance, the system with $\NN=\cb{1,2,3}$ and $P_1=\cb{1,2}, P_2=\cb{1,3}, P_3=\cb{2,3}$ corresponds to a $2$-out-of-$3$ system which functions as long as at least two of the three components function. Our main result provides a formula to calculate the number of coherent systems with $n$ distinct components and $k$ minimal path sets. Alternatively, the configuration of the components in a coherent system can also be encoded by the so-called \emph{minimal cut sets} $C_1,\ldots,C_k$, which are supposed to form a constructive $k$-cover of $\NN$ as in the case of minimal path sets. However, the interpretation of the sets is different in this case: for each $i\in\cb{1,\ldots,k}$, the components in $C_i$ are interpreted as parallel components and the sub-systems $C_1,\ldots,C_k$ are connected in series. Hence, our result also provides the number of coherent systems with $n$ components and $k$ minimal cut sets.

In the present paper, the analysis of the problem is based on a main formula in Section~\ref{problemdefn} (see Theorem~\ref{mainthm} and Corollary~\ref{compcor}) which partitions the set of constructive $k$-covers into certain Cartesian products of pairs of sets. In each pair, the cardinality of the first set is a variant of Stirling numbers of the second kind. This new variant, which we call as \emph{integrated Stirling number (ISN)}, is introduced in Section~\ref{prelim} separately, along with its basic properties. The second set in each Cartesian product gives rise to an auxilary problem which may also be of independent interest: suppose that we are about to form $k\in\N$ sets $A_1,\ldots,A_k$ which partition the world into $2^k$ (disjoint) regions. We consider all the regions except for the one denoting the intersection of all $k$ sets ($A_1\cap\ldots\cap A_k$) and the one denoting the complement of the union of all $k$ sets ($\bar{A}_1\cap\ldots\cap \bar{A}_k$). Given a number $\ell\in\cb{1,\ldots,2^k-2}$, how many ways are there to label these $2^k-2$ regions as ``non-empty" and ``empty" so that there are exactly $\ell$ regions with label ``non-empty" and each set difference $A_i\setminus A_j$, where $i,j\in\cb{1,\ldots,k}$, $i\neq j$, contains at least one region with label ``non-empty"? Note that this new problem is free of $n$, the cardinality of $\NN$, but it depends on $k$, the number of subsets in the constructive cover, as well as the auxiliary variable $\ell$ which keeps track of the regions that are labeled as ``non-empty." While calculating ISNs is an easy task, the above labeling problem is highly nontrivial and requires further analysis. As a result, when using our method, the level of difficulty in counting the number of constructive $k$-covers of $\mathcal{N}$ is largely determined by the value of $k$ rather than $n=|\mathcal{N}|$. In terms of the reliability theory application described above, this implies that our method of counting possible coherent systems is much more sensitive to the number of minimal path (or cut) sets than it is to the number of components.

We carry out a detailed analysis of the auxiliary labeling problem in Section~\ref{partitioningsec}. In particular, we exploit three types of symmetries: based on permutations (Section~\ref{permsym}), taking complements (Section~\ref{compsym}), and the so-called \emph{impact sets} (Section~\ref{impactsym}). These symmetries enable us to define an equivalence relation (Section~\ref{equiv1}) and prove finer results to calculate the answers to the auxiliary problem as well as the original problem. Finally, in Section~\ref{nogoodsec}, we describe the computational procedure based on solving feasibility problems iteratively using \emph{no-good cuts} for binary variables, namely, inequalities that ensure us, in each iteration, to find a solution that is different from the ones found in the former iterations. Some numerical results are also presented in Section~\ref{nogoodsec}. The connection between constructive covers and reliability theory, which was our initial motivation to study the subject, is explained in Section~\ref{reliability}. Section~\ref{conc} is the concluding section. The elementary proofs of the results in Section~\ref{prelim} are collected in Section~\ref{app}, the appendix.

\section{Preliminaries}\label{prelim}

We recall the definition of Stirling numbers of the second kind and introduce \emph{integrated Stirling numbers (ISN)}, a new variant of the former that will be used in the main results of Section~\ref{problemdefn}.

Let $\NN$ be a set with $n\in\N\coloneqq\cb{1,2,\ldots}$ distinct elements. Let $\ell\in\N$. A finite (ordered) sequence $\of{B_1,\ldots,B_\ell}$ of disjoint nonempty subsets of $\NN$ is said to be an \emph{ordered $\ell$-partition} of $\NN$ if
\[
B_1\cup\ldots\cup B_\ell = \NN.
\]
Let $S(\NN,\ell)$ be the set of all ordered $\ell$-partitions of $\NN$.

For $n,\ell$ as above, the corresponding \emph{Stirling number (of the second kind)} is defined as
\begin{equation}\label{StirlingDefn}
s(n,\ell)\coloneqq \frac{1}{\ell!}\sum_{j=0}^\ell (-1)^{\ell-j}\binom{\ell}{j}j^n.
\end{equation}
Note that $\ell!s(n,\ell)$ gives the number of ordered $\ell$-partitions of $\NN$, that is,
\[
|S(\NN,\ell)|=\ell!s(n,\ell).
\]
It can be checked by induction that
\begin{equation}\label{Stirlingzero}
s(n,\ell)=0
\end{equation}
for every $n<\ell$. Moreover, it is well-known that these numbers satisfy the recurrence relation
\[
s(n+1,\ell)=\ell s(n,\ell)+s(n,\ell-1),\quad n\in\N\sm\cb{1},\ \ell\in\cb{2,\ldots,n}
\]
with the boundary conditions
\[
s(n,n)=1,\quad s(n,1)=1,\ n\in\N.
\]

For each $n\in\N$ and $\ell\in\N$, we define the ISN as
\[
\tilde{s}(n,\ell)\coloneqq \frac{1}{\ell!}\sum_{j=0}^\ell (-1)^{\ell-j}\binom{\ell}{j}(j+1)^n.
\]
The next proposition will be used to interpret $\tilde{s}(n,\ell)$ in terms of the number of ordered $\ell$-partitions of subsets of $\NN$. Its proof is given in Section \ref{app}, the appendix.

\begin{proposition}\label{workhorse}
	Let $n,\ell\in\N$. We have
	\begin{equation}\label{stilderes}
	\tilde{s}(n,\ell)=\sum_{i=1}^n\binom{n}{i}s(i,\ell).
	\end{equation}
	In particular, the following are valid.
	\begin{enumerate}[(i)]
		\item If $n<\ell$, then $\tilde{s}(n,\ell)=0$.
		\item If $n\geq \ell$, then
		\begin{equation}\label{stildeinterpret}
		\tilde{s}(n,\ell)=\sum_{i=\ell}^n\binom{n}{i}s(i,\ell)=\sum_{i=0}^{n-\ell}\binom{n}{i}s(n-i,\ell).
		\end{equation}
		\item $\tilde{s}(n,n)=1$.
		\item $\tilde{s}(n,1)=2^n-1$.
		\end{enumerate}
	\end{proposition}

\begin{remark}\label{interpretation}
	Let $1\leq\ell\leq n$. The relation \eqref{stildeinterpret} in Proposition~\ref{workhorse} provides the following characterization of ISNs. For each $i\in\cb{\ell,\ldots,n}$, as mentioned above, $\ell! s(i,\ell)$ is the number of ordered $\ell$-partitions of a set of $i$ distinguishable objects; hence, $\binom ni \ell! s(i,\ell)$ gives the total number of ordered $\ell$-partitions of all subsets of $\mathcal{N}$ with size $i$. Let $\tilde{S}(\NN,\ell)$ be the set of all ordered $\ell$-partitions of all subsets of $\NN$ (with at least size $\ell$), that is,
	\begin{equation}\label{defnstilde}
	\tilde{S}(\NN,\ell)\coloneqq \cb{S(\mathcal{I},\ell)\mid \mathcal{I}\subseteq\NN,|\mathcal{I}|\geq \ell}.
	\end{equation}
	 Therefore,
	 \begin{equation}\label{stildecount}
	 |\tilde{S}(\NN,\ell)|=\ell!\tilde{s}(n,\ell).
	 \end{equation}
\end{remark}

\begin{proposition}\label{stilderec}
	ISNs satisfy the recurrence relation
	\[
	\tilde{s}(n+1,\ell)=(\ell+1)\tilde{s}(n,\ell)+\tilde{s}(n,\ell-1),\quad n\in\N\sm\cb{1}, \quad \ell\in\cb{2,\ldots,n}.
	\]
	with the boundary conditions
	\[
	\tilde{s}(n,n)=1,\ \tilde{s}(n,1)=2^n-1, \ n\in\N.
	\]
	\end{proposition}

Using the recurrence relation in \eqref{stilderec}, we calculate $\tilde{s}(n,\ell)$ for $n\in\cb{1,\ldots,10}$ and $\ell\in\cb{1,\ldots,n}$ as shown in Table~\ref{table2}. For completeness, we also provide the values of $s(n,\ell)$ for the same $(n,\ell)$ pairs in Table~\ref{table1}.

\begin{table}[h]
	\centering
	\begin{tabular}[htbp]{|c|*{10}{p{1cm}|}}\hline 
		\diagbox{$n$}{$\ell$} &$1$  & $ 2$ &  $ 3$ & $ 4$ & $5$& $6$ & $ 7$ & $8$ & $9$ & $10$\\ 	
		\hline 
		$ 1$&  $1$   \\ \cline{1-3} 
		$ 2$&  $1$  & $1$ \\ \cline{1-4}
		$3$&  $ 1$ &$3$ &$1$ \\ \cline{1-5}
		$4$& $1$ &$7$	 & $6$& $1$  \\\cline{1-6}
		$5$& $1$  &$15$&$25$ &$10$ &$1$  \\\cline{1-7}
		$6$& $1$  &$31$&$90$ &$65$ &$15$ & $1$ \\\cline{1-8}
		$7$& $1$  &$63$&$301$ &$350$ &$140$ & $21$ &$1$ \\ \cline{1-9}
		$8$& $1$  &$127$&$966$ &$1701$ &$1050$ & $266$ & $28$ &$1$ \\ \cline{1-10}
		$9$& $1$  &$255$&$3025$ &$7770$ &$6951$ & $2646$ & $462$ &$36$& $1$ \\ \cline{1-11}
		$10$& $1$ &$511$&$9330$ &$34105$ &$42525$ & $22827$ & $5880$ &$750$ & $45$& $1$ \\ \hline 
	\end{tabular}
	\caption{$s(n,\ell)$ values for $1\leq \ell\leq n\leq 10$}\label{table1}
		\centering
		\vspace{0.3cm}
	\begin{tabular}{|c|*{10}{p{1cm}|}}\hline 
		\diagbox{$n$}{$\ell$} &$1$  & $ 2$ &  $ 3$ & $ 4$ & $5$& $6$ & $ 7$ & $8$ & $9$ & $10$\\ 	
		\hline 
		$ 1$&  $1$   \\ \cline{1-3} 
		$ 2$&  $3$  & $1$ \\ \cline{1-4}
		$3$&  $ 7$ &$6$ &$1$ \\ \cline{1-5}
		$4$& $15$ &$25$	 & $10$& $1$  \\\cline{1-6}
		$5$& $31$  &$90$&$65$ &$15$ &$1$  \\\cline{1-7}
		$6$& $63$  &$301$&$350$ &$140$ &$21$ & $1$ \\\cline{1-8}
		$7$& $127$  &$966$&$1701$ &$1050$ &$266$ & $28$ &$1$ \\ \cline{1-9}
		$8$& $255$  &$3025$&$7770$ &$6951$ &$2646$ & $462$ & $36$ &$1$ \\ \cline{1-10}
		$9$& $511$  &$9330$&$34105$ &$42525$ &$22827$ & $5880$ & $750$ &$45$& $1$ \\ \cline{1-11}
		$10$& $1023$ &$28501$&$145750$ &$246730$ &$179487$ & $63987$ & $11880$ &$1155$ & $55$& $1$ \\ \hline 
	\end{tabular}
	\caption{$\tilde{s}(n,\ell)$ values for $1\leq \ell\leq n\leq 10$}\label{table2}
\end{table}

\section{Constructive covers and sets of labelings}\label{problemdefn}

As in the previous section, we consider a set $\NN$ of $n\in\N$ distinct elements. Let $k\in\cb{1,\ldots,n}$.

\begin{definition}\label{kcover}
	A finite sequence $\A=\of{A_1,\ldots,A_k}$ of distinct subsets of $\NN$ is said to be a constructive ordered $k$-cover of $\NN$ if it satisfies the following conditions.
\begin{enumerate}[(i)]
	\item The sets in $\A$ cover $\NN$, that is, $\bigcup_{i=1}^k A_i= \NN$.
	\item Two distinct sets in $\A$ are not subsets of each other, that is, $A_i\sm A_j\neq\emptyset$ for every $i,j\in\cb{1,\ldots,k}$ with $i\neq j$.
\end{enumerate}
\end{definition}

Let $C(\NN,k)$ be the set of all distinct constructive $k$-covers of $\NN$. Our aim is to provide a characterization of the set $C(\NN,k)$ that also helps computing its cardinality $ |C(\NN,k)|$.

\begin{remark}\label{unordered}
	Note that one can also define a constructive \emph{unordered} $k$-cover of $\NN$ as a collection $\cb{A_1,\ldots,A_k}$ of distinct subsets of $\NN$ satisfying the two conditions above. Clearly, the number of all distinct constructive unordered $k$-covers of $\NN$ is $|C(\NN,k)|/k!$.
	\end{remark}

Let $\A=\of{A_1,\ldots,A_k}$ be a constructive ordered $k$-cover of $\NN$. Consider a binary index vector $t=(t_1,\ldots,t_k)\in\cb{e,c}^k$ ($e$ for ``excluded," $c$ for ``contained") and define the set
\[
B_t(\A)\coloneqq \of{\bigcap_{i\colon t_i=c}A_i} \cap\of{\bigcap_{i\colon t_i=e}\bar{A}_i},
\]
where $\bar{A}$ denotes the complement of a subset $A$ of $\NN$. Hence, $\A$ gives rise to $2^k$ disjoint (possibly empty) sets $B_t(\A)$, $t\in\cb{e,c}^k$.

\begin{example}
	To be more specific, let $k=5$ and $t=(e,e,c,e,c)$. Then,
	\[
	B_t(\A)=\bar{A}_1\cap \bar{A}_2\cap A_3\cap \bar{A}_4\cap A_5.
	\]
\end{example}

Two cases need special attention in the above construction. The special vector $\mathbf{e}\coloneqq(e,\ldots,e)$ corresponds to
\[
B_{\mathbf{e}}(\A)=\bar{A}_1\cap\ldots \cap \bar{A}_k,
\]
which must be equal to the empty set since $\A$ is a constructive ordered $k$-cover. On the other hand, the special vector $\mathbf{c}\coloneqq(c,\ldots,c)$ corresponds to
\[
B_{\mathbf{c}}(\A)=A_1\cap\ldots \cap A_k
\]
and the definition of constructive ordered $k$-cover does not impose any non-emptiness condition on this set. Next, we introduce the index sets
\begin{equation}\label{tk}
\T(k)\coloneqq \cb{e,c}^k \sm \cb{\mathbf{e},\mathbf{c}},\quad \T^\ast(k) \coloneqq \cb{e,c}^k\sm \cb{\mathbf{e}},
\end{equation}
and rewrite the conditions (i) and (ii) in Definition~\ref{kcover} as follows.
\begin{enumerate}[(i)]
	\item $\bigcup_{i=1}^k A_i=\bigcup_{t\in\T^\ast(k)}B_t(\A)=\NN$.
	\item For every $i,j\in\cb{1,\ldots,k}$ with $i\neq j$,
	\[
	A_i\sm A_j = \bigcup_{t\in\T(k)\colon t_i=c, t_j=e}B_t(\A)\neq \emptyset.
	\]
\end{enumerate}
Note that condition (ii) depends on the sets $B_t(\A), t\in\T(k)$, only through the \emph{non-emptiness} of certain unions of these sets. Indeed, for each $t\in\T(k)$, let us introduce the binary number
\[
x_t(\A)\coloneqq \begin{cases}1&\text{ if }B_t(\A)\neq\emptyset,\\ 0&\text{ if }B_t(\A)=\emptyset.\end{cases}
\]
Hence, condition (ii) is equivalent to having 
\[
\sum_{t\in\T(k)\colon t_i=c, t_j=e}x_t(\A)\geq 1
\]
for every $i,j\in\cb{1,\ldots,k}$ with $i\neq j$. Let us define
\begin{equation}\label{defng}
G(k)\coloneqq\cb{(x_t)_{t\in\T(k)}\mid \sum_{t\in\T(k)\colon t_i=c, t_j=e}x_t\geq 1\ \forall i\neq j,\ x_t\in\cb{0,1}\ \forall t\in\T(k)}.
\end{equation}
Since $\A$ is a constructive ordered $k$-cover of $\NN$, we necessarily have $(x_t(\A))_{t\in\T(k)}\in G(k)$. Let us also define for each $\ell\in\cb{1,\ldots, 2^{k}-2}$ the set
\begin{equation}\label{DefnOfF}
F(k,\ell)\coloneqq\cb{(x_t)_{t\in\T(k)}\in G(k)\mid\sum_{t\in\T(k)}x_t=\ell},
\end{equation}
which is the set of all labelings in $G(k)$ where exactly $\ell$ sets are labeled as ``non-empty." In short, we refer to $F(k,\ell)$ as the set of all $(k,\ell)$-labelings. We have
\begin{equation}\label{disjointunion}
G(k)=\bigcup_{\ell=1}^{2^k-2}F(k,\ell),
\end{equation}
which is a disjoint union. Hence, we obtain
\[
|G(k)|=\sum_{\ell=1}^{2^k-2}|F(k,\ell)|.
\]

In the following example, we illustrate the notation and the structure of the sets defined above.
\begin{example}
Let us consider the case $k=3$. We have
	\begin{equation}\label{tkorder}
	\T(3)=\cb{(c,e,e),(e,c,e),(e,e,c),(c,c,e),(c,e,c),(e,c,c)}.
	\end{equation}
	Then, we can write the corresponding set $G(3)$ as
	\begin{align*}
	G(3)=\Big\{(x_t)_{t\in\T(3)}\mid \ & x_{(c,e,e)}+x_{(c,e,c)}\geq 1,\ x_{(e,c,e)}+x_{(e,c,c)}\geq 1,\ x_{(c,e,e)}+x_{(c,c,e)}\geq 1,\\
	&x_{(e,e,c)}+x_{(e,c,c)}\geq 1,\ x_{(e,c,e)}+x_{(c,c,e)}\geq 1,\ x_{(e,e,c)}+x_{(c,e,c)}\geq 1,\\
	& x_t\in\cb{0,1}\ \forall t\in\T(3)\Big\}.
	\end{align*}
	Hence, for each $\ell\in\cb{1,\ldots,6}$,
	\begin{align*}
	F(3,\ell)=\Big\{(x_t)_{t\in\T(3)}\mid \ & x_{(c,e,e)}+x_{(c,e,c)}\geq 1,\ x_{(e,c,e)}+x_{(e,c,c)}\geq 1,\ x_{(c,e,e)}+x_{(c,c,e)}\geq 1,\\
	&x_{(e,e,c)}+x_{(e,c,c)}\geq 1,\ x_{(e,c,e)}+x_{(c,c,e)}\geq 1,\ x_{(e,e,c)}+x_{(c,e,c)}\geq 1,\\
	&x_{(e,e,c)}+x_{(e,c,e)}+x_{(e,c,c)}+x_{(c,e,e)}+x_{(c,e,c)}+x_{(c,c,e)}=\ell
	\Big\}.
	\end{align*}
	We observe that $F(3,1)=F(3,2)=\emptyset$. Further, we can explicitly express $F(3,3)$ as
	\begin{align*}
	&F(3,3)=\cb{(c,c,c,e,e,e),(e,e,e,c,c,c)},\\
	&F(3,4)=\{(c,e,e,c,c,c),(e,c,e,c,c,c),(e,e,c,c,c,c),(c,c,c,c,e,e),(c,c,c,e,c,e),(c,c,c,e,e,c),\\
		&\quad\quad\quad \quad \quad  (c,c,e,e,c,c),(c,e,c,c,e,c),(e,c,c,c,c,e)\},\\
	&F(3,5)=\cb{(c,c,c,c,c,e),(c,c,c,c,e,c),(c,c,c,e,c,c),(c,c,e,c,c,c),(c,c,e,c,c,c),(e,c,c,c,c,c)},\\
	&F(3,6)=\cb{(c,c,c,c,c,c)},
	\end{align*}
	following the order in \eqref{tkorder}. Hence,
	\[
	G(3)=F(3,3)\cup F(3,4)\cup F(3,5)\cup F(3,6).
	\]
	\end{example}

Using \eqref{disjointunion}, we may write
\[
C(\NN,k)=\bigcup_{\ell=1}^{2^{k}-2}\cb{\A\in C(\NN,k)\mid (x_t(\A))_{t\in\T(k)}\in F(k,\ell)}
\]
as a disjoint union so that
\begin{equation}\label{subthm}
|C(\NN,k)|=\sum_{\ell=1}^{2^k-2}|\cb{\A\in C(\NN,k)\mid (x_t(\A))_{t\in\T(k)}\in F(k,\ell)}|.
\end{equation}

The next theorem is the main result of the paper. Up to isomorphisms, it characterizes $C(\NN,k)$ as a disjoint union of Cartesian products of basic sets. For two sets $E_1,E_2$, we write $E_1\cong E_2$ if there is a bijection $f\colon E_1\to E_2$.

\begin{theorem}\label{mainthm}
	For each $\ell\in\cb{1,\ldots,2^k-2}$, it holds
		\begin{equation}\label{congell}
	\cb{\A\in C(\NN,k)\mid (x_t(\A))_{t\in\T(k)}\in F(k,\ell)}\cong\tilde{S}(n,\ell) \times F(k,\ell).
	\end{equation}
	In particular,
	\begin{equation}\label{cong}
	C(\NN,k)\cong\bigcup_{\ell=1}^{2^k-2}\tilde{S}(n,\ell) \times F(k,\ell).
	\end{equation}
	\end{theorem}

\begin{proof}
	Let $\ell\in\cb{1,\ldots,2^k-2}$. Let $\A=(A_1,\ldots,A_k)\in C(\NN,k)$ with $(x_t(\A))_{t\in\T(k)}\in F(k,\ell)$. In other words, $\A$ is a constructive ordered $k$-cover of $\NN$ for which exactly $\ell$ of the sets $B_t(\A)$, $t\in\T(k)$, are nonempty. Denoting by $\prec$ the strict lexicographical ordering on $\T(k)$, let us order these $\ell$ indices as $t^1(\A)\prec\ldots\prec t^\ell(\A)$. Then, $\mathcal{B}(\A)\coloneqq (B_{t^1(\A)}(\A),\ldots,B_{t^\ell(\A)}(\A))$ is an ordered $\ell$-partition of $\mathcal{N}\sm B_{\mathbf{c}}(\A)$. Hence, $\mathcal{B}(\A)\in S(\mathcal{N}\sm B_{\mathbf{c}}(\A),\ell)\subseteq \tilde{S}(\NN,\ell)$ by \eqref{defnstilde}.
	
	The above construction establishes the mapping
	\begin{equation}\label{mapping}
	\A\mapsto (\mathcal{B}(\A),(x_t(\A))_{t\in\T(k)})
	\end{equation}
	from $\cb{\A\in C(\NN,k)\mid (x_t(\A))_{t\in\T(k)}\in F(k,\ell)}$ to $\tilde{S}(\NN,\ell)\times F(k,\ell)$. To check that this mapping is injective, let $\A^\prime=(A^\prime_1,\ldots,A^\prime_k)\in C(\NN,k)$ be another constructive ordered $k$-cover such that $(x_t(\A^\prime))_{t\in\T(k)}\in F(k,\ell)$. Suppose that
	\[
	(x_t(\A))_{t\in\T(k)}=(x_t(\A^\prime))_{t\in\T(k)},\quad \mathcal{B}(\A)=\mathcal{B}(\A^\prime).
	\]
	The first supposition guarantees that $\A$ and $\A^\prime$ agree on the nonemptiness of their corresponding sets $B_t(\A), B_t(\A^\prime)$ for each $t\in\T(k)$. In other words,
	\[
	B_t(\A)=\emptyset\quad\Leftrightarrow\quad B_t(\A^\prime)=\emptyset
	\]
	for each $t\in\T(k)$. Moreover, from the definition of lexicographical ordering, it follows that
	\[
	t^1\coloneqq t^1(\A)=t^1(\A^\prime),\ldots, t^\ell\coloneqq t^\ell(\A)=t^\ell(\A^\prime).
	\]
	Then, by the second supposition, we have
	\[
	B_{t^1}(\A)=B_{t^1}(\A^\prime),\ldots, B_{t^\ell}(\A)=B_{t^\ell}(\A^\prime).
	\]
	Hence,
	\[
	B_t(\A)=B_t(\A^\prime)
	\]
	for every $t\in\T(k)$. Since $\bigcup_{i=1}^k A_i = \bigcup_{i=1}^k A^\prime_i=\NN$, we also have
	\[
	B_{\mathbf{c}}(\A)=B_{\mathbf{c}}(\A^\prime).
	\]
	Finally, we have
	\[
	A_i= \of{\bigcup_{t\in\T(k)\colon t_i=c}B_t(\A)}\cup B_{\mathbf{c}}(\A)= \of{\bigcup_{t\in\T(k)\colon t_i=c}B_t(\A^\prime)}\cup B_{\mathbf{c}}(\A^\prime)=A^\prime_i
	\] 
	for every $i\in\cb{1,\ldots,k}$ so that $\A=\A^\prime$. This finishes the proof of injectivity. 
	
	Next, we show that the mapping in \eqref{mapping} is surjective. Let $\mathcal{B}=(\bar{B}_1,\ldots,\bar{B}_\ell)\in \tilde{S}(\NN,\ell)$ and $(x_t)_{t\in\T(k)}\in F(k,\ell)$. Hence, by \eqref{defnstilde} in Remark~\ref{interpretation}, there exists $\mathcal{I}\subseteq\NN$ with $|\mathcal{I}|\geq \ell$ such that $\mathcal{B}\in S(\mathcal{I},\ell)$. Let us set
	\[
	B_{\mathbf{c}}\coloneqq \NN\sm\of{\bar{B}_1\cup\ldots\cup \bar{B}_{\ell}}=\NN\sm\mathcal{I}.
	\]
	On the other hand, consider the set of all $t\in\T(k)$ for which $x_t =1$. Since $(x_t)_{t\in\T(k)}\in F(k,\ell)$, there are $\ell$ such indices in $\T(k)$. As before, let us order them as $t^1\prec\ldots\prec t^\ell$ using the lexicographical ordering and set
	\[
	B_{t^1}\coloneqq \bar{B}_1,\ldots, B_{t^\ell}\coloneqq \bar{B}_\ell
	\]
	and
	\[
	B_t\coloneqq \emptyset
	\]
	for every $t\in\T(k)\sm\cb{t^1,\ldots,t^\ell}$. Then, let
	\[
	A_i\coloneqq \of{\bigcup_{t\in\T(k)\colon t_i=c}B_t}\cup B_{\mathbf{c}}
	\]
	for each $i\in\cb{1,\ldots,k}$ and $\A\coloneqq (A_1,\ldots,A_k)$. It is clear that $\bigcup_{i=1}^k A_i=\NN$. Moreover, the assumption that $(x_t)_{t\in\T(k)}\in F(k,\ell)\subseteq G(k)$ guarantees that $A_i\sm A_j\neq\emptyset$ for every $i,j\in\cb{1,\ldots,k}$ with $i\neq j$. Finally, by construction of the mapping in \eqref{mapping}, we conclude that $\mathcal{B}(\A)=\mathcal{B}$ and $(x_t(\A))_{t\in\T(k)}=(x_t)_{t\in\T(k)}$. This shows that every element of $\tilde{S}(\NN,\ell)\times F(k,\ell)$ is the value of the mapping in \eqref{mapping} for some $\A\in C(\NN,\ell)$ with $(x_t(\A))_{t\in\T(k)}\in F(k,\ell)$. 
	
	Therefore,  \eqref{congell} follows. As an immediate consequence of disjointness, \eqref{cong} holds as well.
	\end{proof}

\begin{corollary}\label{compcor}
	It holds
	\[
	|C(\NN,k)|=\sum_{\ell=1}^{(2^k-2)\wedge n}\ell!\tilde{s}(n,\ell) |F(k,\ell)|.
	\]
\end{corollary}

\begin{proof}
By \eqref{stildecount} in Remark~\ref{interpretation}, $|\tilde{S}(n,\ell)\times F(k,\ell)|=\ell!\tilde{s}(n,\ell) |F(k,\ell)|$ for each $\ell\in\cb{1,\ldots,2^k-2}$. Hence, the corollary follows from Theorem~\ref{mainthm}.
\end{proof}

Next, we aim to refine the result of Corollary~\ref{compcor} by showing that $F(k,\ell)$ is the empty set for small values of $\ell$. To that end, for a subset $T\subseteq\T(k)$ and $i,j\in\cb{1,\ldots,k}$ with $i\neq j$, let us define
\[
u_{ij}(T)\coloneqq\sum_{t\in T}1_{\cb{c}}(t_i)1_{\cb{e}}(t_j).
\]

\begin{proposition}\label{ell0result}
	Let
	\[
	\ell_0(k)\coloneqq \min\cb{\abs{T}\mid u_{ij}(T)\geq 1\ \forall i\neq j,\ T\subseteq\T(k)}
	\]
	Then, for every $\ell\in\cb{1,\ldots,2^k-2}$,
	\[
	\ell \geq \ell_0(k)\quad\Leftrightarrow\quad  F(k,\ell)\neq \emptyset.
	\]
	In particular,
	\[
	\ell_0(k)=\min\cb{\ell\mid F(k,\ell)\neq\emptyset}.
	\]
	\end{proposition}

\begin{proof}
	Let $\ell\in\cb{1,\ldots,2^k-2}$. Suppose that $\ell\geq\ell_0(k)$. Let $T^\ast\subseteq\T(k)$ such that $|T^\ast|=\ell_0(k)$ and $u_{ij}(T^\ast)\geq 1$ for every $i,j\in\cb{1,\ldots,k}$ with $i\neq j$. By adding $\ell-\ell_0$ more elements to $T^\ast$ arbitrarily, one can find a set $T\subseteq\T(k)$ such that $T^\ast\subseteq T$ and $|T|=\ell$. For each $t\in\T(k)$, let us define a binary variable $x_t$ by $x_t=1$ if $t\in T$ and $x_t=0$ if $T\in\T(k)\sm T$. The assumed properties of $T^\ast$ ensure that $(x_t)_{t\in\T(k)}\in F(k,\ell)$. Hence, $F(k,\ell)\neq \emptyset$. In particular, the case $\ell=\ell_0$ implies that $\ell_0(k)\geq \min\cb{\ell\mid F(k,\ell)\neq\emptyset}$.
	
For the converse, suppose that $F(k,\ell)\neq \emptyset$. Let $(x_t)_{t\in\T(k)}\in F(k,\ell)$ and $i,j\in\cb{1,\ldots,k}$ with $i\neq j$. By the definition of $F(k,\ell)$, there exists $t(i,j)\in \T(k)$ such that $x_{t(i,j)}=1$, $t(i,j)_i=c$, $t(i,j)_j=e$. Let
	\[
	T=\cb{t(i,j)\mid i\neq j}.
	\]
	Note that
	\[
	u_{ij}(T)=\sum_{t\in T}1_{\cb{c}}(t_i)1_{\cb{e}}(t_j)\geq 1_{\cb{c}}(t(i,j)_i)1_{\cb{e}}(t(i,j)_j)=1\cdot 1=1.
	\]
	Hence, $\ell_0(k)\leq |T|$.	By the definition of $F(k,\ell)$ again,
	\[
	\ell_0(k)\leq |T|\leq \sum_{i\neq j}x_{t(i,j)}\leq \sum_{t\in\T(k)}x_t=\ell.
	\]
	Hence, $\ell_0(k)\leq \ell$. In particular, $\ell_0(k)\leq \min\cb{\ell\mid F(k,\ell)\neq\emptyset}$.
	\end{proof}

\begin{corollary}\label{ell0forcounting}
	It holds
	\[
	|C(\NN,k)|=\sum_{\ell=\ell_0(k)}^{(2^k-2)\wedge n}\ell!\tilde{s}(n,\ell) |F(k,\ell)|.
	\]
	\end{corollary}

\begin{proof}
	This is an immediate consequence of Corollary~\ref{compcor} and Proposition~\ref{ell0result}.
	\end{proof}

\section{Partitioning the sets of labelings}\label{partitioningsec}
	

By Theorem~\ref{mainthm} and Corollary~\ref{ell0forcounting} of Section~\ref{problemdefn}, we are able to calculate the cardinality of the set $C(\NN, k)$ of constructive $k$-covers of $\NN$ in terms of ISNs $\tilde{s}(n,\ell)$ as well as the cardinalities of the sets $F(k,\ell)$ of $(k,\ell)$-labelings for a range of $\ell$ values. While it is easy to numerically calculate ISNs by \eqref{stilderes} and {\eqref{StirlingDefn}}, the calculation of $|F(k,\ell)|$ by brute force enumeration could be quite difficult even for small values of $k,\ell$. In this section, we introduce three notions of symmetry for the sets $F(k,\ell)$, $\ell\in\cb{\ell_0(k),\ldots,2^k-2}$, which yield an equivalence relation. It turns out that the equivalence classes of this relation provide substantial reduction in the computational effort to find the cardinalities $|F(k,\cdot)|$. 

Let $k\in\N$. Let us fix a nonempty subset $T$ of $\T(k)$. We call $T$ the \emph{branching set} of index vectors. Let us define
\[
\mathcal{Z}_T(k)\coloneqq \cb{y=(y_t)_{t\in T}\mid y_t\in\cb{0,1}\ \forall t\in T}.
\]
We also fix $\ell\in\N$ with $\ell_0(k)\leq \ell\leq 2^k-2$, where $\ell_0(k)$ is defined as in Proposition~\ref{ell0result}. In particular, $F(k,\ell)\neq\emptyset$. We may consider partitioning $F(k,\ell)$ with respect to the possible ways of assigning the binary variables $x_t$ associated to all $t\in T$. The next definition formalizes this idea.

\begin{definition}\label{DefnOfFm}
Let $y_t\in\cb{0,1}$ for each $t\in T$. Then, the set of all $(k,\ell)$-labelings with respect to $y= (y_t)_{t\in T}$ is defined as 
\[
F_{y}(k,\ell)\coloneqq\cb{(x_t)_{t\in\T(k)}\in F(k,\ell)\mid x_t = y_t \ \forall t\in T}.
\]
\end{definition}

If $\sum_{t\in T}y_t >\ell$, then $F_{y}(k,\ell)=\emptyset$ obviously. Let us introduce the set
\[
\mathcal{Y}_T(k,\ell)\coloneqq \cb{y=(y_t)_{t\in T}\in\mathcal{Z}_T(k)\mid \sum_{t\in T}y_t \leq \ell}.
\]

\begin{remark}
Given $y\in \Y_T(k,\ell)$, depending on the structure of $T$ and $y$, the set $F_y(k,\ell)$ may still be empty. Nevertheless, such cases will be detected in the computational procedure presented later in this section and we do not need to distinguish them \emph{a priori} in the theoretical development.
\end{remark}

We begin with a simple result that provides a partitioning of $F(k,\ell)$ into smaller sets.

\begin{proposition}\label{PropBranch}
Let $y,z\in\Y_T(k,\ell)$.
\begin{enumerate}[(i)]
	\item $F_{y}(k, \ell) \cap F_{z}(k, \ell) = \emptyset$ if $y\neq z$.
	\item It holds
	\[
	F(k,\ell)=\bigcup_{y\in \Y_T(k,\ell)}F_{y}(k, \ell).
	\]
\end{enumerate}
In particular,
\[
|F(k, \ell)| = \sum_{y\in \Y_T(k,\ell)} |F_y(k, \ell)|.
\]
\end{proposition}
\begin{proof}
	\text{}
\begin{enumerate}[(i)]
\item This is an immediate consequence of Definition~\ref{DefnOfFm}.
\item The $\supseteq$ part of the equality is obvious since $F_{y}(k, \ell) \subseteq F(k, \ell)$ for each $y\in \Y_T(k,\ell)$. For the $\subseteq$ part, let $(x_t)_{t\in\T(k)}\in F(k, \ell)$. Let us define $y=(y_t)_{t\in T}$ by setting $y_t\coloneqq x_t$ for each $t\in T$. Then, $\sum_{t\in T}y_t= \sum_{t\in T}x_t\leq\sum_{t\in\T(k)}x_t=\ell$ so that $y\in\Y_T(k,\ell)$. Clearly, we also have $(x_t)_{t\in \T(k)}\in F_{y}(k,\ell)$. Hence, the $\subseteq$ part of the equality follows.
\end{enumerate}
The last statement follows directly from (i) and (ii).
\end{proof}

While Proposition~\ref{PropBranch} partitions $F(k,\ell)$ into the smaller sets $F_y(k,\ell)$, $y\in\Y_T(k,\ell)$, it may still be computationally expensive to calculate the cardinality of each of these sets by an enumerative method. By introducing three notions of symmetry below, we show that the cardinalities $|F_y(k,\ell)|$ are repeated for many $y\in\Y_T(k,\ell)$ so that we only need to calculate the distinct values of these cardinalities and the number of times each cardinality value is repeated.

\subsection{A symmetry based on permutations}\label{permsym}

The first notion of symmetry we introduce is based on permutations. To that end, let us denote by $\mathbb{S}_k$ the symmetric group of $\cb{1,\ldots,k}$, that is, the set of all permutations $\pi\colon\cb{1,\ldots,k}\to\cb{1,\ldots,k}$. Let $\pi\in\S_k$. Given $t=(t_1,\ldots,t_k)\in \cb{e,c}^k$, we may consider $t$ as a function $t\colon\cb{1,\ldots,k}\to \cb{e,c}$ and define the composition $t\circ \pi \colon\cb{1,\ldots, k}\to\cb{e,c}$ by
\[
(t\circ \pi)(i)=(t\circ\pi)_i\coloneqq t_{\pi(i)},\quad i\in\cb{1,\ldots,k},
\]
or, we may simply define $t\circ \pi$ as the vector
\[
t\circ \pi = ((t\circ \pi)_1,\ldots,(t\circ\pi)_k)\coloneqq (t_{\pi(1)},\ldots, t_{\pi(k)})\in\cb{e,c}^k.
\]
From the definition of $\T(k)$ (see \eqref{tk}), it is clear that $t\circ \pi \in \T(k)$ if and only if $t\in\T(k)$. Let us introduce the set
\[
T^\pi \coloneqq \cb{t\circ \pi \mid t\in T}.
\]
We call $T$ symmetric with respect to $\pi$ if $T=T^\pi$. In particular, it is always the case that $\T(k)$ is symmetric with respect to $\pi$. We denote by $\S_k^T$ the set of all permutations with respect to which $T$ is symmetric, that is,
\[
\S_k^T \coloneqq\cb{\pi\in\S_k\mid T=T^\pi}.
\]
Note that $\S_k^T\neq\emptyset$ as we always have $T=T^\pi$ when $\pi$ is the identity permutation.

The next proposition formulates how the number of $(k,\ell)$-labelings associated to a vector $y\in\Y_T(k,\ell)$ changes under the application of a permutation.

\begin{proposition}\label{equalityofcard}
Let $\pi\in\S_k^T$ and $y\in \Y_T(k,\ell)$. Define $y^\pi=(y^\pi_t)_{t\in T}$ by
\[
y^\pi_t \coloneqq y_{t\circ \pi},\quad t\in T.
\]
Then, $y^\pi \in \Y_T(k,\ell)$ and $|F_y(k,\ell)|=|F_{y^\pi}(k,\ell)|$.
\end{proposition}

\begin{proof}
Since $y\in\Y_T(k,\ell)$ and $T=T^\pi$, we have
\[
\sum_{t\in T}y^\pi_t = \sum_{t\in T}y_{t\circ\pi}=\sum_{t\in T}y_t \leq \ell.
\]
Hence, $y^\pi\in \Y_T(k,\ell)$.

To prove that $|F_y(k,\ell)|=|F_{y^\pi}(k,\ell)|$, it is sufficient to establish a bijection from $F_y(k,\ell)$ into $F_{y^\pi}(k,\ell)$. To that end, given $(x_t)_{t\in \T(k)}\in F_y(k,\ell)$, let us define $(x^\pi_t)_{t\in \T(k)}$ by
\[
x^\pi_t = x_{t\circ \pi},\quad t\in \T(k).
\]
We first show that $(x^\pi_t)_{t\in\T(k)}\in F_{y^\pi}(k,\ell)$. Let $i,j\in\cb{1,\ldots,k}$ with $i\neq j$. Denoting the inverse permuatation of $\pi$ by $\pi^{-1}$, we have
\begin{align*}
\sum_{t\in\T(k)\colon t_i=e,t_j=c}x^\pi_t
&= \sum_{t\in\T(k)\colon t_i=e,t_j=c}x_{t\circ \pi}\\
&=\sum_{t\in\T(k)\colon (t\circ\pi)_{\pi^{-1}(i)}=e,(t\circ\pi)_{\pi^{-1}(j)}=c}x_{t\circ \pi}\\
&=\sum_{t\in\T(k)\colon t_{\pi^{-1}(i)}=e,t_{\pi^{-1}(j)}=c}x_{t},
\end{align*}
where we make a change of variables using the fact that $\T(k)=(\T(k))^\pi$ in order to get the last equality. Since $i\neq j$, we have $\pi^{-1}(i)\neq \pi^{-1}(j)$. As we also have $(x_t)_{t\in\T(k)}\in G(k)$ (see \eqref{defng}), it follows that
\[
\sum_{t\in\T(k)\colon t_i=e,t_j=c}x^\pi_t=\sum_{t\in\T(k)\colon t_{\pi^{-1}(i)}=e,t_{\pi^{-1}(j)}=c}x_{t}\geq 1.
\]
Similarly, since $\T(k)=(\T(k))^\pi$,
\[
\sum_{t\in\T(k)}x^\pi_t = \sum_{t\in \T(k)}x_{t\circ\pi}=\sum_{t\in\T(k)}x_t = \ell.
\]
On the other hand, for each $t\in T$, we have $t\circ\pi\in T$ so that
\[
x^\pi_t = x_{t\circ \pi} = y_{t\circ \pi}=y^\pi_t.
\]
Therefore, $(x^\pi_t)_{t\in\T(k)}\in F_{y^\pi}(k,\ell)$.

It remains to check that the mapping $(x_t)_{t\in\T(k)}\mapsto (x_t^\pi)_{t\in\T(k)}$ is indeed a bijection from $F_y(k,\ell)$ into $F_{y^\pi}(k,\ell)$. Let $(x_t)_{t\in \T(k)},$ $(\bar{x}_t)_{t\in\T(k)}\in F_y(k,\ell)$ such that $x_t^\pi =\bar{x}_t^\pi$ for every $t\in\T(k)$, that is, $x_{t\circ \pi}=\bar{x}_{t\circ\pi}$ for every $t\in\T(k)$. Since $\T(k)=(\T(k))^\pi$, this is equivalent to having $x_t = \bar{x}_t$ for every $t\in\T(k)$. Hence, $(x_t)_{t\in\T(k)}\mapsto (x_t^\pi)_{t\in\T(k)}$ is injective. Next, let $(z_t)_{t\in\T(k)}\in F_{y^\pi}(k,\ell)$. Let us define $(x_t)_{t\in\T(k)}$ by
\[
x_t \coloneqq z^{\pi^{-1}}_t=z_{t\circ \pi^{-1}},\quad t\in\T(k).
\]
Let $i,j\in\cb{1,\ldots,k}$ such that $i\neq j$. Hence, we have $\pi(i)\neq\pi(j)$ so that
\[
\sum_{t\in\T(k)\colon t_i=e,\ t_j=c}x_t=\sum_{t\in\T(k)\colon t_i=e,\ t_j=c} z_{t\circ \pi^{-1}}=\sum_{t\in\T(k)\colon t_{\pi(i)}=e,\ t_{\pi(j)}=c} z_{t}\geq 1.
\]
Next, since $\T(k)=(\T(k))^{\pi^{-1}}$, we have
\[
\sum_{t\in\T(k)}x_t = \sum_{t\in\T(k)}z_{t\circ \pi^{-1}}=\sum_{t\in\T(k)}z_t = \ell.
\]
On the other hand, since $T=T^\pi$, we also have $T=T^{\pi^{-1}}$. Hence, for each $t\in T$, we have $t\circ\pi^{-1}\in T$ so that
\[
x_t = z_{t\circ\pi^{-1}}=y_{t\circ \pi^{-1}}^{\pi}=y_t.
\]
Therefore, $(x_t)_{t\in\T(k)}\in F_y(k,\ell)$, that is, $(x_t)_{t\in\T(k)}\mapsto (x_t^\pi)_{t\in\T(k)}$ is surjective as well.
\end{proof}

\subsection{A symmetry based on taking complements}\label{compsym}

In addition to the above notion of permutation-based symmetry, we introduce a second type of symmetry based on taking ``complements,'' that is, based on changing the roles of $e$ and $c$ in the index vectors. To be more precise, let us define two mappings $\alpha_1,\alpha_2\colon\cb{e,c}\to\cb{e,c}$ by
\[
\alpha_1(e)=\alpha_2(c)=e,\quad \alpha_1(c)=\alpha_2(e)=c.
\]
In other words, $\alpha_1$ is the identity mapping and $\alpha_2$ switches $e$ and $c$. Let us formally define the set $\mathbb{A}=\cb{\alpha_1,\alpha_2}$, which is indeed the symmetric group of $\cb{e,c}$. Let $\alpha\in\mathbb{A}$. Similar to what is done in Section \ref{permsym}, we may regard each $t=(t_1,\ldots,t_k)\in\cb{e,c}^k$ as a function $t\colon\cb{1,\ldots,k}\to\cb{e,c}$ and define the composition $\alpha\circ t\colon \cb{1,\ldots,k}\to\cb{e,c}$ by
\[
(\alpha\circ t)(i)=(\alpha\circ t)_i \coloneqq \alpha(t_i),\quad i\in\cb{1,\ldots,k},
\]
or, we define $\alpha\circ t$ as the vector
\[
\alpha\circ t = \of{(\alpha\circ t)_1,\ldots,(\alpha\circ t)_k}\coloneqq \of{\alpha(t_1),\ldots,\alpha(t_k)}\in \cb{e,c}^k.
\]
Clearly, $\alpha\circ t\in\T(k)$ if and only if $t\in \T(k)$. Let us also define
\[
T^\alpha \coloneqq \cb{\alpha\circ t\mid t\in T}
\]
and
\[
\mathbb{A}^T \coloneqq \cb{\alpha\in \mathbb{A}\mid T=T^\alpha}.
\]
Since $\alpha_1$ is the identity mapping, we always have $T=T^{\alpha_1}$ so that $\mathbb{A}^T \neq \emptyset$.

In the next proposition, we relate the numbers of $(k,\ell)$-labelings associated to a vector $y\in\Y_T(k,\ell)$ before and after taking complements.

\begin{proposition}\label{equalityofcard2}
	Let $\alpha\in \mathbb{A}^T$ and $y\in \Y_T(k,\ell)$. Define ${}^\alpha y=({}^\alpha y_t)_{t\in T}$ by
	\[
	{}^\alpha y_t \coloneqq y_{\alpha\circ t},\quad t\in T.
	\]
	Then, ${}^\alpha y\in \Y_T(k,\ell)$ and $|F_y(k,\ell)|=|F_{{}^\alpha y}(k,\ell)|$.
	\end{proposition}

\begin{proof}
The result is trivial for $\alpha=\alpha_1$. Let us assume that $\alpha=\alpha_2$. Since $y\in\Y_T(k,\ell)$ and $T=T^{\alpha}$, we have
	\[
	\sum_{t\in T}{}^\alpha y_t = \sum_{t\in T}y_{\alpha\circ t}=\sum_{t\in T}y_t \leq \ell.
	\]
	Hence, ${}^\alpha y\in \Y_T(k,\ell)$.

	To prove that $|F_y(k,\ell)|=|F_{{}^\alpha y}(k,\ell)|$, we construct a bijection from $F_y(k,\ell)$ to $F_{{}^\alpha y}(k,\ell)$ as follows. Given $(x_t)_{t\in\T(k)}\in F_y(k,\ell)$, let us define $({}^\alpha x_t)_{t\in\T(k)}$ by
	\[
	{}^\alpha x_t =x_{\alpha \circ t},\quad t\in\T(k).
	\]
	Let $i,j\in\cb{1,\ldots,k}$ with $i\neq j$. Since $(x_t)_{t\in \T(k)}\in G(k)$, we have
	\begin{align*}
	\sum_{t\in\T(k)\colon t_i=e,t_j=c}{}^\alpha x_t &= \sum_{t\in\T(k)\colon t_i=e,t_j=c}x_{\alpha \circ t}\\
	&= \sum_{t\in\T(k)\colon \alpha(t_i)=c,\alpha(t_j)=e}x_{\alpha \circ t}\\
	&=\sum_{t\in\T(k)\colon t_i=c,t_j=e}x_t\geq 1.
	\end{align*}
	Similarly,
	\[
	\sum_{t\in\T(k)}{}^\alpha x_t= \sum_{t\in\T(k)}x_{\alpha \circ t}=\sum_{t\in\T(k)}x_t = \ell
	\]
	and 
	\[
	{}^\alpha x_t = x_{\alpha \circ t}=y_{\alpha \circ t}={}^\alpha y_t,\quad t\in T.
	\]
	Hence, $({}^\alpha x_t)_{t\in\T(k)}\in F_{{}^\alpha y}(k,\ell)$. Using a similar argument as in the proof of Proposition \ref{equalityofcard}, it can be shown that the mapping $(x_t)_{t\in\T(k)}\mapsto ({}^\alpha x_t)_{t\in\T(k)}$ is a bijection. The details are omitted.
	\end{proof}

\subsection{A symmetry based on impact sets}\label{impactsym}

In this subsection, we introduce a third notion of symmetry based on the idea that two binary vectors $y,z\in\mathcal{Z}_T(k)$ might impose the same set of constraints in the definition of $G(k)$, see \eqref{defng}, which we will refer to as the \emph{impact set} of these vectors.

Given $y=(y_t)_{t\in T}\in\mathcal{Z}_T(k)$, we define the \emph{impact set} $\D_T(y)$ of $y$ as
\[
\D_T(y)\coloneqq \cb{(i,j)\in\cb{1,\ldots,k}^2\mid \exists t\in T\colon (t_i=c\ \wedge\ \ t_j=e\ \wedge\  y_t=1)}.
\]
The next theorem provides a relationship for the values of $|F_y(k,\cdot)|$ and $|F_z(k,\cdot)|$ when $y,z\in\mathcal{Z}_T(k)$ have the same impact set.

\begin{theorem}\label{simeqprop}
	Let $y,z\in \mathcal{Z}_T(k)$ be such that $\D_T(y)=\D_T(z)$. Let
	\[
	w \coloneqq \sum_{t\in T}y_t - \sum_{t\in T}z_t.
	\]
	Suppose that $w\geq 0$. The following results hold.
	\begin{enumerate}[(i)]
		\item $|F_y(k,\ell+w)|=|F_z(k,\ell)|$ for every $\ell\in\N$ such that $\ell_0(k)\leq \ell \leq \ell+w\leq 2^k-2$.
		\item $|F_y(k,\ell+w)|=0$ for every $\ell\in\N$ such that $\ell< \ell_0(k)\leq \ell+w$.
		\item $|F_z(k,\ell)|=0$ for every $\ell\in\N$ such that $\ell\leq  2^k-2 <\ell+w$.
	\end{enumerate}
\end{theorem}

\begin{proof}
	We consider the sets $F_y(k,\ell), F_z(k,\ell)$ defined by \eqref{DefnOfF} for every $\ell\in\N$. (Hence, we extend the definition in \eqref{DefnOfF} for $\ell\geq 2^k-2$). Let us fix $\ell\in\N$. We establish a bijection from $F_z(k,\ell)$ to $F_y(k,\ell+w)$. Given $x=(x_t)_{t\in\T(k)}\in F_z(k,\ell)$, let us define $\bar{x}=(\bar{x}_t)_{t\in\T(k)}$ by
	\[
	\bar{x}_t = \begin{cases}y_t&\text{if } t\in T,\\ x_t & \text{if }t\notin T.\end{cases}
	\]
	Let $i,j\in\cb{1,\ldots,k}$ with $i\neq j$. First, suppose that $(i,j)\in \D_T(y)=\D_T(z)$. Hence, there exists $t^1\in T$ such that $t^1_i=c$, $t^1_j=e$, $y_{t^1}=1$. So
	\[
	\sum_{t\in\T(k)\colon t_i=c,t_j=e}\bar{x}_t\geq \bar{x}_{t^1}=y_{t^1}=1.
	\]
	Next, suppose that $(i,j)\notin \D_T(y)=\D_T(z)$. Since $x\in F_z(k,\ell)$, there exists $t^2\in \T(k)\setminus T$ such that $t^2_i=c$, $t^2_j=e$, $x_{t^2}=1$. So
	\[
	\sum_{t\in\T(k)\colon t_i=c,t_j=e}\bar{x}_t\geq \bar{x}_{t^2}=x_{t^2}=1.
	\]
	Moreover,
	\[
	\sum_{t\in \T(k)}\bar{x}_t = \sum_{t\in T}y_t +\sum_{t\in\T(k)\setminus T}x_t =\sum_{t\in T}y_t +\ell - \sum_{t\in T}x_t=\sum_{t\in T}y_t +\ell - \sum_{t\in T}z_t=\ell+w.
	\]
	Hence, $\bar{x}\in F_y(k,\ell+w)$.
	
	Next, we show that the mapping $x\mapsto \bar{x}$ is a bijection from $F_z(k,\ell)$ into $F_y(k,\ell+w)$. Let $x^1=(x^1_t)_{t\in\T(k)}, x^2=(x^2_t)_{t\in \T(k)}\in F_z(k,\ell)$ such that their images are equal, that is, $\bar{x}^1_t = \bar{x}^2_t$ for every $t\in\T(k)$. From the definition of the mapping, it is immediate that $x^1_t=x^2_t$ for every $t\in \T(k)\setminus T$. On the other hand, $x^1_t=x^2_t=z_t$ for every $t\in T$ since $x^1,x^2\in F_z(k,\ell)$. Therefore, $x^1=x^2$. This proves that the mapping is injective. Let $\tilde{x}=(\tilde{x}_t)_{t\in\T(k)}\in F_y(k,\ell+w)$. Define $x=(x_t)_{t\in\T(k)}$ by
	\[
	x_t=\begin{cases}z_t&\text{if }t\in T,\\ \tilde{x}_t&\text{if }t\notin T.\end{cases}
	\]
	It is not difficult to check that $x\in F_z(k,\ell)$ and $\bar{x}=\tilde{x}$, which shows that the mapping is surjective. Hence, thanks to the bijection, we conclude that $|F_y(k,\ell+w)|=|F_z(k,\ell)|$ for every $\ell\in\N$. If $\ell<\ell_0(k)$, then $F_z(k,\ell)=\emptyset$ so that $F_y(k,\ell+w)=\emptyset$ as well. Similarly, if $2^k-2<\ell+w $, then $F_y(k,\ell+w)=\emptyset$ so that $F_z(k,\ell)=\emptyset$ as well. Hence, all three results hold.
\end{proof}

\subsection{Equivalence relation for the symmetries}\label{equiv1}

Given $t\in \T(k)$, $x_t\in\cb{e,c}$, $\pi\in\S_k$ and $\alpha\in\mathbb{A}$, note that 
\[
({}^\alpha x_t)^\pi = {}^\alpha (x^\pi_t)=x_{\alpha\circ t\circ \pi}.
\]
Hence, we simply write ${}^\alpha x^\pi_t\coloneqq x_{\alpha\circ t\circ \pi}$ for the common value. On $\mathcal{Z}_T(k)$, let us define the relation $\equiv$ by
\begin{equation}\label{equivalence}
y\equiv z\quad \Leftrightarrow\quad \exists(\pi,\alpha)\in\S_k^T\times\mathbb{A}^T\colon \D_T(z)=\D_T({}^\alpha y^\pi)
\end{equation}
for each $y=(y_t)_{t\in T},z=(z_t)_{t\in T}\in\mathcal{Z}_T(k)$.
\begin{proposition}\label{equivrel}
The relation $\equiv$ defined by \eqref{equivalence} is an equivalence relation on $\mathcal{Z}_T(k)$. 
	\end{proposition}
\begin{proof}
	The reflexivity of $\equiv$ follows since the identity permutation on $\cb{1,\ldots,k}$ is always a member of $\S_k^T$. To show that $\equiv$ is symmetric, let $y,z\in \mathcal{Z}_T(k)$ such that $y\equiv z$. So $\D_T(z)=\D_T({}^\alpha y^\pi)$ for some $\pi\in\S_k^T$ and $\alpha\in\mathbb{A}^T$. Then,
	\[
	T^{\pi^{-1}}=\cb{t\circ \pi^{-1}\mid t\in T}=\cb{\bar{t}\circ\pi\circ\pi^{-1}\mid \bar{t}\in T}=T
	\]
	since $T=T^{\pi}$. So $\pi^{-1}\in\S_k^T$. On the other hand, it is easy to see that $\alpha^{-1}=\alpha$. We claim that $\D_T(y)=\D_T({}^\alpha z^{\pi^{-1}})$. To show this, first, let $(i,j)\in\D_T(y)$. So there exists $t\in T$ such that $t_i=c$, $t_j=e$, $y_t=1$. Letting $t^\prime\coloneqq \alpha\circ t\circ \pi^{-1}$, we may write $t=\alpha\circ t^\prime\circ \pi$. In particular, we have the following:
	\begin{enumerate}[(i)]
		\item $t^\prime_{\pi(i)}=\alpha(t_{\pi^{-1}(\pi(i))})=\alpha(t_i)=\alpha(c)$.
		\item $t^{\prime}_{\pi(j)}=\alpha(t_{\pi^{-1}(\pi(j))})=\alpha(t_j)=\alpha(e)$.
		\item ${}^\alpha y_{t^\prime}^{\pi}=y_{\alpha\circ t^{\prime}\circ \pi}=y_t=1$.
	\end{enumerate}
	It follows that $(\pi(i),\pi(j))\in\D_T({}^\alpha y^{\pi})$ if $\alpha=\alpha_1$, and $(\pi(j),\pi(i))\in\D_T({}^\alpha y^{\pi})$ if $\alpha=\alpha_2$. Let us consider the case $\alpha=\alpha_1$. Since $(\pi(i),\pi(j))\in\D_T({}^\alpha y^{\pi})=\D_T(z)$, there exists $t^{\prime\prime}\in T$ such that $t^{\prime\prime}_{\pi(i)}=c$, $t^{\prime\prime}_{\pi(j)}=e$, $z_{t^{\prime\prime}}=1$. Let us define $t^{\prime\prime\prime}\coloneqq \alpha\circ t^{\prime\prime}\circ \pi$ so that $t^{\prime\prime}=\alpha\circ t^{\prime\prime\prime}\circ \pi^{-1}$. In particular, we have the following:
		\begin{enumerate}[(i)]
		\item $t^{\prime\prime\prime}_{i}=\alpha(t^{\prime\prime}_{\pi(i)})=\alpha(c)=c$.
		\item $t^{\prime\prime\prime}_{j}=\alpha(t^{\prime\prime}_{\pi(j)})=\alpha(e)=e$.
		\item ${}^\alpha z_{t^{\prime\prime\prime}}^{\pi^{-1}}=z_{\alpha\circ t^{\prime\prime\prime}\circ \pi^{-1}}=z_{t^{\prime\prime}}=1$.
	\end{enumerate}
	Hence, $(i,j)\in \D_T({}^\alpha z^{\pi^{-1}})$. The case $\alpha=\alpha_2$ can be treated by similar arguments and we obtain $(i,j)\in \D_T({}^\alpha z^{\pi^{-1}})$ as well. Hence, $\D_T(y)\subseteq \D_T({}^\alpha z^{\pi^{-1}})$.
	
	Conversely, let $(i,j)\in \D_T({}^\alpha z^{\pi^{-1}})$. So there exists $t\in T$ such that $t_i=c$, $t_j=e$, ${}^\alpha z^{\pi^{-1}}_t=1$. Let $t^\prime\coloneqq\alpha\circ t \circ \pi^{-1}$. In a similar way as above, we have $t^\prime_{\pi(i)}=\alpha(c)$, $t^{\prime}_{\pi(j)}=\alpha(e)$, $z_{t^\prime}=1$, that is, $(\pi(i),\pi(j))\in \D_T(z)$ if $\a=\a_1$, and $(\pi(j),\pi(i))\in \D_T(z)$ if $\a=\a_2$. Suppose that $\a=\a_1$. In this case, since $(\pi(i),\pi(j))\in \D_T(z)=\D_T({}^\alpha y^\pi)$, there exists $t^{\prime\prime}\in T$ such that $t^{\prime\prime}_{\pi(i)}=c$, $t^{\prime\prime}_{\pi(j)}=e$, ${}^\a y^\pi_{t^{\prime\prime}}=1$. Let $t^{\prime\prime\prime}\coloneqq \alpha\circ t^{\prime\prime}\circ \pi$. Then, we have $t^{\prime\prime\prime}_i=\a(c)=c$, $t^{\prime\prime\prime}_j=\a(e)=e$, $y_{t^{\prime\prime\prime}}=1$ so that $(i,j)\in \D_T(y)$. Similarly, we may show that $(i,j)\in \D_T(y)$ in the case $\a=\a_2$ as well. Hence, $\D_T({}^\alpha z^{\pi^{-1}})\subseteq\D_T(y)$. This completes the proof of $\D_T(y)=\D_T({}^\alpha z^{\pi^{-1}})$. Therefore, $z\equiv y$.
	
	To show that $\equiv$ is transitive, let $y,z,v\in \mathcal{Z}_T(k)$ such that $y\equiv z$ and $z\equiv v$. So $\D_T(z)=\D_T({}^\alpha y^\pi)$ and $\D_T(v)=\D_T({}^\beta z^{\sigma})$ for some $\pi,\sigma\in\S_k^T$ and $\alpha,\beta\in\mathbb{A}^T$. Since $T^\pi=T^{\sigma}=T$, we also have $T^{\sigma\circ \pi}=\cb{t\circ\sigma\circ \pi\mid t\in T}=\cb{\bar{t}\circ \pi \mid \bar{t}\in T}=T$ so that $\sigma\circ\pi\in\S_k^T$. On the other hand, $\alpha\circ\beta $ is either equal to $\alpha_1$ or to $\alpha_2$ so that $\alpha\circ\beta\in\mathbb{A}^T$. We claim that $\D_T(v)=\D_T({}^{\alpha\circ\beta}y^{\sigma\circ\pi})$.
	
	To prove the claim, let $(i,j)\in \D_T(v)$. Since $\D_T(v)=\D_T({}^\beta z^{\sigma})$, there exists $t\in T$ such that $t_i=c$, $t_j=e$, ${}^\beta z^\sigma_t=1$. Letting $t^\prime=\beta\circ t\circ\sigma$, we have $t^\prime_{\sigma^{-1}(i)}\coloneqq\beta(c)$, $t^\prime_{\sigma^{-1}(j)}=\beta(e)$, $z_{t^\prime}=1$. Hence, $(\sigma^{-1}(i),\sigma^{-1}(j))\in\D_T(z)$ if $\beta=\a_1$, and $(\sigma^{-1}(j),\sigma^{-1}(i))\in\D_T(z)$ if $\beta=\a_2$. Suppose that $\beta=\a_1$. Since $(\sigma^{-1}(i),\sigma^{-1}(j))\in\D_T(z)=\D_T({}^\alpha y^\pi)$, there exists $t^{\prime\prime}\in T$ such that $t^{\prime\prime}_{\sigma^{-1}(i)}=c$, $t^{\prime\prime}_{\sigma^{-1}(j)}=e$, ${}^\alpha y^\pi_{t^{\prime\prime}}=1$. Letting $t^{\prime\prime\prime}=\beta\circ t^{\prime\prime}\circ\sigma^{-1}$, we have $t^{\prime\prime\prime}_i=c$, $t^{\prime\prime\prime}_i=e$, ${}^{\alpha\circ\beta}y^{\sigma\circ\pi}_{t^{\prime\prime\prime}}={}^{\a}y^{\pi}_{t^{\prime\prime}}=1$. Hence, $(i,j)\in\D_T({}^{\alpha\circ\beta}y^{\sigma\circ\pi})$. Similarly, we can reach the same conclusion when $\b=\a_2$. So $\D_T(v)\subseteq \D_T({}^{\alpha\circ\beta}y^{\sigma\circ\pi})$. The proof of $ \D_T({}^{\alpha\circ\beta}y^{\sigma\circ\pi})\subseteq \D_T(v)$ is similar, hence we omit it. So we have $y\equiv v$. Therefore, $\equiv$ is an equivalence relation on $\mathcal{Z}_T(k)$.
	\end{proof}

Next, we address the role of Proposition~\ref{equivrel} for computational purposes. Note that the relation $\equiv$ partitions $\mathcal{Z}_T(k)$ into equivalence classes; let us denote them by $\mathcal{Z}_{T,1}(k),\ldots,\mathcal{Z}_{T,A}(k)$, where $A\in\N$ is the number of distinct classes. Algorithm \ref{alg1} shows the precise steps of the main procedure in which these classes are calculated. The two subroutines of this procedure that are used to find the equivalent elements with respect to $\equiv$ are given in Algorithm \ref{alg2} and Algorithm \ref{alg3}. Let $a\in\cb{1,\ldots,A}$ and let $z^{T,a}$ be a fixed representative element of $\mathcal{Z}_{T,a}(k)$. For each $y\in\mathcal{Z}_{T,a}(k)$, the relationship between $|F_{y}(k,\cdot)|$ and $|F_{z^{T,a}}(k,\cdot)|$ is formulated by Theorem \ref{simeqprop}: indeed, letting
\[
w(y,a)\coloneqq \sum_{t\in T}y_{t}-\sum_{t\in T}z^{T,a}_t,
\]
we have
\begin{align}
|F_{y}(k,\ell)|=\begin{cases}
|F_{z^{T,a}}(k,\ell-w(y,a))|&\text{if }\ell,\ell-w(y,a)\in\cb{\ell_0(k),\ldots,2^{k}-2},\\ 
0 &\text{else}.\end{cases}\label{shift}
\end{align}
Therefore, the values $|F_y(k,\cdot)|$ for a given $y\in\mathcal{Z}_{T,a}(k)$ can be calculated by a ``shift" of the values $|F_{z^{T,a}}(k,\cdot)|$ according to \eqref{shift}. This will be illustrated further in Example \ref{k=4ex} and Example \ref{k=5ex} in the next section.

Let us introduce some additional notation that will make the presentation of the examples simpler. First, without loss of generality, we assume that $z^{T,a}$ achieves the minimum possible sum $\sum_{t\in T}y_t$ among all $y\in \mathcal{Z}_{T,a}(k)$; hence $w(y,a)\geq 0$ for each $y\in\mathcal{Z}_{T,a}(k)$ and $w(z^{T,a},a)=0$. Let $\W(a)$ be the set of all possible values of $w(y,a)$, that is, 
\[
\W(a)\coloneqq\cb{w(y,a)\mid y\in\mathcal{Z}_{T,a}(k)},
\]
which is a finite subset of $\Z_+=\cb{0,1,2,\ldots}$. For each $w\in\W(a)$, let us define
\[
\mathcal{Z}_{T,a,w}(k)=\cb{y\in\mathcal{Z}_{T,a}(k)\mid w(y,a)=w}.
\]
Hence, we may partition $\mathcal{Z}_{T,a}(k)$ as
\[
\mathcal{Z}_{T,a}(k)=\bigcup_{w\in\W(a)}\mathcal{Z}_{T,a,w}(k).
\]
We fix one representative $z^{T,a,w}\in \mathcal{Z}_{T,a,w}(k)$ for each $w\in\W(a)$ such that $z^{T,a,0}=z^{T,a}$. Hence, in view of Proposition \ref{PropBranch}, we may write
\[
|F(k,\ell)|=\sum_{a=1}^A \sum_{w\in\W(a)}|F_{z^{T,a,w}}(k,\ell)|\cdot|\mathcal{Z}_{T,a,w}(k)|
\]
for each $\ell\in\cb{\ell_0(k),\ldots,2^k-2}$. We will use these notation and reformulations in the examples of the next section.


	 \algnewcommand{\NULL}{\textsc{null}}
\algnewcommand{\algvar}{\texttt}
\begin{algorithm}
	\caption{Main Algorithm}\label{alg1}
	\begin{algorithmic}[1]
		\State $\bar{\mathcal{D}}_{T} \gets \emptyset $ 
		\For{$y$ \textbf{in} $\mathcal{Z}_{T}(k)$}
		\State $\bar{\mathcal{D}}_{T} \gets \bar{\mathcal{D}}_{T}\cup \mathcal{D}_{T}(y) $
		\EndFor
		\State $asymmetricMatrices \gets \emptyset$
		\State $A \gets 0$
		\For{$y$ \textbf{in} $\mathcal{Z}_{T}(k)$} \Comment{Finding Equivalence Classes}
		\State $symmetryMatrix \gets \algvar{generateSymmetryMatrix}(\bar{\mathcal{D}}_{T}, y, T)$
		\State $a \gets \algvar{findEquivalenceClass}(symmetryMatrix, asymmetricMatrices, \bar{\mathcal{D}}_{T} )$
		\If{$a == \NULL$}
		\State $A\gets A+1$
		\State $\mathcal{Z}_{T,A}(k)\gets \{y\}$
		\State $z^{T,A}\gets y$
		\State $asymmetricMatrices \gets asymmetricMatrices \cup \cb{symmetryMatrix}$
		\Else
		\State $\mathcal{Z}_{T,a}(k)\gets \mathcal{Z}_{T,a}(k)\cup\{y\}$
		\EndIf
		\EndFor
		\For{$a=1:A$}
		\State $\W(a) \gets \emptyset$
		\For{$y$ \textbf{in} $\mathcal{Z}_{T,a}(k)$}
		\State $w \gets \sum_{t\in T}y_t-\sum_{t\in T}z^{T,a}_{t}$
		\If{$w$ \textbf{in} $\W(a)$}
		\State  $\mathcal{Z}_{T, a, w} \gets \mathcal{Z}_{T, a, w} \cup \{y\}$
		\Else
		\State $\W(a) \gets \W(a) \cup \{ w \}$
		\State $\mathcal{Z}_{T, a, w} \gets \{ y \}$
		\State $z^{T,a,w} \gets y$
		\EndIf
		\EndFor
		\EndFor
		\For{$a=1:A$}
		\For{$\ell=1:(2^k-2)$} 
		\State Calculate $|F_{z^{T,a}}(k,\ell)|$ by no-good cuts.\Comment{see text}
		\EndFor 
		\EndFor 
		\State $|F(k,\ell)|\gets 0$
		\For{$a=1:A$}\Comment{Total Number of Solutions}
		\For{$w$ \textbf{in} $\W(a)$  }
		\For{$\ell=\ell_0(k):(2^k-2)$} 
		\If{$\ell\in\{\ell_0(k),\ldots,2^{k}-2\}\textbf{ and }\ell-w\in\cb{\ell_0(k),\ldots,2^k-2}$}
		\State $|F(k,\ell)|\gets|F(k,\ell)|+|F_{z^{T,a}}(k,\ell-w)|\cdot|\mathcal{Z}_{T,a,w}(k)|$
		\EndIf 
		\EndFor 
		\EndFor
		\EndFor
	\end{algorithmic}
\end{algorithm}

\begin{algorithm}
	\caption{\algvar{generateSymmetryMatrix}($\bar{\mathcal{D}}_{T}, y, T$)}\label{alg2}
	\begin{algorithmic}[1]
		\State $symmetryMatrix \gets zeros(|\bar{\mathcal{D}}_{T}|, |\mathbb{S}^T_k|\cdot|\mathbb{A}^T|)$
		\State $A \gets 1$
		\For{$(i,j)$ \textbf{in} $\bar{\mathcal{D}}_{T}$}
		\State $B \gets 1$
		\For{$\pi$ \textbf{in} $\mathbb{S}^T_k$}
		\For{$\alpha$ \textbf{in} $\mathbb{A}^T$}
		\For{$t$ \textbf{in} $T$}
		\If{$(\alpha\circ t\circ \pi)_i == c$ \textbf{and} $(\alpha\circ t\circ \pi)_j == e$ \textbf{and} ${}^\alpha y^\pi_t == 1$}
		\State $symmetryMatrix(A,B) = 1$
		\EndIf
		\EndFor
		\State $B \gets B + 1$
		\EndFor
		\EndFor
		\State $A \gets A + 1$
		\EndFor
		\State \Return symmetryMatrix
	\end{algorithmic}
\end{algorithm}

\begin{algorithm}
	\caption{\algvar{findEquivalenceClass}($symmetryMatrix$, $asymmetricMatrices, \bar{\mathcal{D}}_{T}$)}\label{alg3}
	\begin{algorithmic}[1]
		\State $A \gets 1$
		\For{$M$ \textbf{in} $asymmetricMatrices$}\Comment{\parbox[t]{.5\linewidth} {Checking $\D_T(z)=\D_T({}^\alpha y^\pi)$ for any $\alpha$ and $\pi$ by dot product of their corresponding vectors. All the column vectors inside $symmetryMatrix$ represent $\D_T({}^\alpha y^\pi)$ for different choices of $(\alpha,\pi)$.}} 
		\If{$sum(M) == sum(symmetryMatrix)$ \ldots\\
			\quad \quad \quad \ldots\;  \textbf{and} ($sum(M)/|\bar{\mathcal{D}}_{T}| $ \textbf{exists in} $(M^T * symmetryMatrix)$)} 
		\State \Return $A$            
		\EndIf
		\State $A \gets A+1$
		\EndFor
		\State \Return \NULL
	\end{algorithmic}
\end{algorithm}

\subsection{Computational procedure with no-good cuts}\label{nogoodsec}

With the equivalence relation $\equiv$ introduced in Section \ref{equiv1}, we calculate the cardinalities $|F(k,\cdot)|$ through the equivalence classes $\mathcal{Z}_{T,a}(k)$, $a\in\cb{1,\ldots,A}$, and the corresponding $|F_{z^{T,a}}(k,\cdot)|$ values for a fixed representative $z^{T,a}$ of each class.

It remains to calculate the number of solutions for each equivalence class (as needed in line 21 of Algorithm \ref{alg1}). To that end, let $a\in\cb{1,\ldots,A}$ and $z=z^{T,a}$. In order to calculate $|F_z(k,\ell)|$, we adopt an optimization approach based on the so-called \emph{no-good cuts}. Given a set $F^\ast\subseteq F_z(k,\ell)$, we consider the following integer-linear optimization problem:
\begin{align}
\text{maximize}& \qquad 0\tag{$\mathscr{P}(z,F^\ast)$} \\
\text{subject to}& \qquad \sum_{t\in\T(k)\colon t_i=c, t_j=e}x_t\geq 1\quad \forall i\neq j\ \label{constraint1} \\
\emph{}& \qquad \sum_{t\in\T(k)}x_t=\ell \label{constraint2} \\
\emph{}&\qquad x_t = y_t \quad \forall t\in T\label{constraint5}\\
\emph{}& \qquad \sum_{t\in\T(k)\colon x^*_t=1} (1-x_t) + \sum_{t\in\T(k)\colon x^*_t=0}x_t \quad \geq 1 \quad \forall (x^*_t)_{t\in\T(k)} \in F^* \label{constraint3} \\
\emph{}& \qquad x_t\in \{0,1\} \quad \forall  t\in\T(k) \label{constraint4}
\end{align}
Note that $(\mathscr{P}(z,F^\ast))$ is basically a feasibility problem since we maximize a constant function. Here, constraints \eqref{constraint1}, \eqref{constraint2}, \eqref{constraint4} are the relations in the definition of $F(k,\ell)$; constraint \eqref{constraint5} is the additional requirement in Definition~\ref{DefnOfFm}. Constraint \eqref{constraint3} is the collection of no-good cuts; it makes sure that the new solution $(x_t)_{t\in\T(k)}$ to be found is different from each of the solutions $(x_t^\ast)_{t\in\T(k)}$ that are already stored in $F^\ast$. Indeed, it is not difficult to check that the inequality in \eqref{constraint3} is equivalent to having
\[
\exists t\in\T(k)\colon x_t\neq x^\ast_t
\]
for each $(x^\ast_t)_{t\in\T(k)}\in F^\ast$.
While more general formulations of no-good cuts for continuous variables are nonconvex inequalities, the formulation we use here is for binary variables due to the nature of our problem and it yields a linear inequality. The reader is referred to \cite{nogoodcut} for a detailed treatment of no-good cuts.

The computational procedure is initialized by solving $(\mathscr{P}(z,F^{(1)}))$ with $F^{(1)}\coloneqq \emptyset$. If $F_z(k,\ell)=\emptyset$, then this problem terminates by infeasibility. Otherwise, it yields a $(k,\ell)$-labeling $(x_t^{(1)})_{t\in\T(k)}$ in $F_z(k,\ell)$ as an optimal solution. For each $u\in\N$, in the $(u+1)^{\text{st}}$ iteration, we call $(\mathscr{P}(z,F^{(u+1)}))$ with $F^{(u+1)}\coloneqq F^{(u)}\cup\{(x_t^{(u)})_{t\in\T(k)}\}$. If $|F_z(k,\ell)|\leq u$, then the problem terminates by infeasibility. Otherwise, it yields a $(k,\ell)$-labeling in $F_z(k,\ell)$ that is different from each of the points in $F^{(u)}$. The procedure terminates by infeasibility when $u=|F_z(k,\ell)|$ in which case we have $F_z(k,\ell) = F^{(u)}\cup\{(x_t^{(u)})_{t\in\T(k)}\}$. Hence, we find the cardinality of $F_z(k,\ell)$ by finding all of its elements.

It should be noted that, as the iteration number increases, the number of no-good cuts in the optimization problem increases and it takes more time to compute a new solution. Hence, when deciding on the choice of $T$, it is desirable to control $|F_z(k,\ell)|$ by a reasonable upper bound that depends on $T$. For this purpose, we use the simple upper bound 
	\[
	|F_z(k,\ell)|\leq \binom{2^k-2-\bar{z}}{\ell-\bar{z}},
	\]
	where $\bar{z}\coloneqq\sum_{t\in T}z_t$; however, better upper bounds can be obtained by exploiting the structure of $T$.
	
	To illustrate the computational procedure, in the ‘‘Total" rows of Table \ref{k=4table1} and Table \ref{k=5table1}, we present the $|F(k,\cdot)|$ values for $k=4$ and $k=5$, respectively. From these tables, it is notable that for fixed $k$, the quantity $|F(k,\ell)|$ first increases with $\ell$, then makes a maximum around $\ell=2^{k-1}-1$ and then decreases with $\ell$. Detailed explanations on these tables are provided in the following two examples.
	
	\begin{example}\label{k=4ex}
		Suppose that $k=4$. In this case, we have $\ell_0(k)=4$ and $2^k-2=14$. Hence, we consider the values $\ell\in\{4,\ldots,14\}$. As the branching set of index vectors, we select $T=\{(c,e,e,e),(c,e,c,e),(c,e,e,c),(c,e,c,e)\}$, which corresponds to the regions $A_1 \cap \bar{A}_2 \cap \bar{A}_3 \cap  \bar{A}_4$, $A_1 \cap \bar{A}_2 \cap A_3 \cap \bar{A}_4$, $A_1 \cap \bar{A}_2 \cap \bar{A}_3 \cap A_4$, $A_1 \cap \bar{A}_2 \cap A_3 \cap A_4$. Hence, $|\mathcal{Z}_T(k)|=16$. By Algorithm~\ref{alg1}, we find out that $\equiv$ has $A=6$ equivalence classes. 
		In the header column of Table~\ref{k=4table1}, we use the format $(a,w,s)$ to report the equivalence class index $a$, the value $w\in\W(a)$ that is fixed for the corresponding row, and the cardinality $s=|\mathcal{Z}_{T,a,w}(k)|$ that corresponds to $(a,w)$. For instance, the row of $(5,1,2)$ gives the $|F_{y}(4,\cdot)|$ values for each of the two members of $\mathcal{Z}_{T,5,1}(4)$. The entries in the ‘‘Total" row represent the values $|F(k,\ell)|=\sum_{a=1}^{6}\sum_{w\in\W(a)}|F_{z^{T,a,w}}(k,\ell)|\cdot|\mathcal{Z}_{T,a,w}(k)|$ for all $\ell\in\cb{\ell_0(k),\ldots,2^k-2}$. In Table~\ref{k=4table1}, the rows corresponding to the same equivalence class are placed consecutively and shaded with the same color; we alternate between two colors as the class index $a$ changes. The representatives $z^{T,a,w}$, $a\in\cb{1,\ldots,A}$, $w\in\W(a)$, are listed in Table~\ref{k=4table2}. For instance, the line for $(5,1)$ states that
		\[
		z^{T,5,1}_{(c,e,e,e)}=1,\quad z^{T,5,1}_{(c,e,c,e)}=1, \quad z^{T,5,1}_{(c,e,e,c)}=1,\quad z^{T,5,1}_{(c,e,c,e)}=0.
		\]
		\end{example}

 \begin{table}[t]
 		\centering
 	\resizebox{\textwidth}{!}{%
	\begin{tabular}
		{|c|*{12}{p{0.8cm}|}}
		\hline 
		\rowcolor{lightgray}
		\rule{-3.5pt}{3.5ex}
		\diagbox{\small{$(a,w,s)$}}{\small{$\ell$}} &$4$ &$5$ &$6$ &$7$ &$8$ &$9$ &$10$ &$11$ &$12$ &$13$ &$14$  \\
		\cline{1-12} 
		
		$\cellcolor{lightgray}(1,0,2)$ &\cellcolor{lightred}$4$ &\cellcolor{lightred}$40$ &\cellcolor{lightred}$115$ &\cellcolor{lightred}$146$ &\cellcolor{lightred}$103$ &\cellcolor{lightred}$43$ &\cellcolor{lightred}$10$ &\cellcolor{lightred}$1$ &\cellcolor{lightred}$0$ &\cellcolor{lightred}$0$ &\cellcolor{lightred}$0$  \\ \cline{1-12}
		
		$\cellcolor{lightgray}(2,0,2)$ &\cellcolor{lightgreen}$6$ &\cellcolor{lightgreen}$47$ &\cellcolor{lightgreen}$120$ &\cellcolor{lightgreen}$147$ &\cellcolor{lightgreen}$103$ &\cellcolor{lightgreen}$43$ &\cellcolor{lightgreen}$10$ &\cellcolor{lightgreen}$1$ &\cellcolor{lightgreen}$0$ &\cellcolor{lightgreen}$0$ &\cellcolor{lightgreen}$0$ \\ \cline{1-12}
		
		$\cellcolor{lightgray}(3,0,4)$ &\cellcolor{lightred}$0$ &\cellcolor{lightred}$14$ &\cellcolor{lightred}$77$ &\cellcolor{lightred}$159$ &\cellcolor{lightred}$172$ &\cellcolor{lightred}$111$ &\cellcolor{lightred}$44$ &\cellcolor{lightred}$10$ &\cellcolor{lightred}$1$ &\cellcolor{lightred}$0$ &\cellcolor{lightred}$0$ \\ \cline{1-12}
		
		$\cellcolor{lightgray}(4,0,1)$ &\cellcolor{lightgreen}$1$ &\cellcolor{lightgreen}$20$ &\cellcolor{lightgreen}$94$ &\cellcolor{lightgreen}$184$ &\cellcolor{lightgreen}$191$ &\cellcolor{lightgreen}$118$ &\cellcolor{lightgreen}$45$ &\cellcolor{lightgreen}$10$ &\cellcolor{lightgreen}$1$ &\cellcolor{lightgreen}$0$ &\cellcolor{lightgreen}$0$ \\ \cline{1-12}
		
		$\cellcolor{lightgray}(5,0,1)$ &\cellcolor{lightred}$4$ &\cellcolor{lightred}$44$ &\cellcolor{lightred}$141$ &\cellcolor{lightred}$222$ &\cellcolor{lightred}$205$ &\cellcolor{lightred}$120$ &\cellcolor{lightred}$45$ &\cellcolor{lightred}$10$ &\cellcolor{lightred}$1$ &\cellcolor{lightred}$0$ &\cellcolor{lightred}$0$ \\ \cline{1-12}
		
		$\cellcolor{lightgray}(5,1,2)$ &\cellcolor{lightred}$0$ &\cellcolor{lightred}$4$ &\cellcolor{lightred}$44$ &\cellcolor{lightred}$141$ &\cellcolor{lightred}$222$ &\cellcolor{lightred}$205$ &\cellcolor{lightred}$120$ &\cellcolor{lightred}$45$ &\cellcolor{lightred}$10$ &\cellcolor{lightred}$1$ &\cellcolor{lightred}$0$ \\ \cline{1-12}
		
		$\cellcolor{lightgray}(5,2,1)$ &\cellcolor{lightred}$0$ &\cellcolor{lightred}$0$ &\cellcolor{lightred}$4$ &\cellcolor{lightred}$44$ &\cellcolor{lightred}$141$ &\cellcolor{lightred}$222$ &\cellcolor{lightred}$205$ &\cellcolor{lightred}$120$ &\cellcolor{lightred}$45$ &\cellcolor{lightred}$10$ &\cellcolor{lightred}$1$  \\ \cline{1-12}
		
		$\cellcolor{lightgray}(6,0,2)$ &\cellcolor{lightgreen}$0$ &\cellcolor{lightgreen}$1$ &\cellcolor{lightgreen}$30$ &\cellcolor{lightgreen}$117$ &\cellcolor{lightgreen}$203$ &\cellcolor{lightgreen}$198$ &\cellcolor{lightgreen}$119$ &\cellcolor{lightgreen}$45$ &\cellcolor{lightgreen}$10$ &\cellcolor{lightgreen}$1$ &\cellcolor{lightgreen}$0$ \\ \cline{1-12}
					\hline
		\rowcolor{lightgray}
		Total &$25$ &$304$ &$1165$ &$2188$ &$2487$ &$1882$ &$989$ &$364$ &$91$ &$14$ &$1$  \\ \cline{1-12}
		\hline 
	\end{tabular}}
	\caption{$|F_{z^{T,a,w}}(4, \ell)|$ values for $\ell_0(k) \leq \ell\leq 2^k-2 $ in Example \ref{k=4ex}} \label{k=4table1}
\end{table}

\begin{table}[t]
	\resizebox{\textwidth}{!}{%
	\begin{tabular}{|>{\columncolor{lightgray}}P{2.2cm}|P{3.2cm}|P{3.2cm}|P{3.2cm}|P{3.2cm}|}
		\hline
		\rowcolor{lightgray}
	    \rule{-4pt}{1ex}
		\diagbox{\small{$(a,w)$}}{\small{$T$}} &$ A_1 \cap \bar{A}_2 \cap \bar{A}_3 \cap  \bar{A}_4$ 
		&$A_1 \cap \bar{A}_2 \cap A_3 \cap \bar{A}_4$
		&$A_1 \cap \bar{A}_2 \cap \bar{A}_3 \cap A_4$
		&$A_1 \cap \bar{A}_2 \cap A_3 \cap A_4$ \\
		\hline
		$(1,0)$ & \cellcolor{lightgreen}$1$ & \cellcolor{lightred}$0$ &\cellcolor{lightred} $0$ & \cellcolor{lightred}$0$ \\
		\hline
		$(2,0)$ & \cellcolor{lightred}$0$ &\cellcolor{lightgreen} $1$ &\cellcolor{lightred} $0$ & \cellcolor{lightred}$0$ \\
		\hline
		$(3,0)$ & \cellcolor{lightgreen}$1$ &\cellcolor{lightgreen} $1$ &\cellcolor{lightred} $0$ &\cellcolor{lightred} $0$  \\
		\hline
		$(4,0)$ & \cellcolor{lightgreen}$1$ & \cellcolor{lightred}$0$ & \cellcolor{lightred}$0$ &\cellcolor{lightgreen} $0$ \\
		\hline
		$(5,0)$ &\cellcolor{lightred} $1$ &\cellcolor{lightgreen} $0$ &\cellcolor{lightgreen} $0$ & \cellcolor{lightred}$0$ \\
		\hline
		$(5,1)$ &\cellcolor{lightgreen}$1$ &\cellcolor{lightgreen} $1$ & \cellcolor{lightgreen}$1$ & \cellcolor{lightred}$0$ \\
		\hline
		$(5,2)$ &\cellcolor{lightgreen} $1$ & \cellcolor{lightgreen}$1$ & \cellcolor{lightgreen}$1$ & \cellcolor{lightgreen}$1$ \\
		\hline
		$(6,0)$ &\cellcolor{lightgreen} $1$ &\cellcolor{lightgreen} $1$ & \cellcolor{lightred}$0$ & \cellcolor{lightgreen}$1$ \\
		\hline
	\end{tabular}
}
	\caption{Definitions of the representatives $z^{T,a,w}$ in Example \ref{k=4ex}}\label{k=4table2}
\end{table}

\begin{example}\label{k=5ex}
	Suppose that $k=5$. In this case, we have $\ell_0(k)=4$ and $2^k-2=30$. As the branching set of index vectors, we fix $T=\cb{t^1,\ldots,t^{15}}$ and the precise definitions of $t^1,\ldots,t^{15}$ are given in Table \ref{k=5table3}. For instance, $t^7=(c,e,c,e,e)$. We have $|\mathcal{Z}_T(k)|=2^{15}$ and there are $A=28$ equivalence classes of $\equiv$. Table \ref{k=5table1} provides the $|F_{z^{T,a,w}}(5,\ell)|$ values and it is presented in two pages. Its structure is the same as that of Table \ref{k=4table1}. The definitions of the representatives $z^{T,a,w}$, $a\in\cb{1,\ldots,28}$, $w\in\W(a)$, are given in Table \ref{k=5table2}.
	
	Table \ref{k=5table1} illustrates the tremendous reduction in computational effort provided by the symmetries encoded in $\equiv$. For each value of $a$, we only compute the values $|F_{z^{T,a,w}}(k,\cdot)|$ for $w=0$, that is, the values in the first row for $a$; then, for $w>0$, the rows are obtained by simple shifts of the row $w=0$. The row of $(a,w)$ is repeated $s=|\mathcal{Z}_{T,a,w}(k)|$ times in the ultimate value of $|F(k,\ell)|$. For instance, in a brute-force calculation without using symmetries, the values in the row for $(27,4)$ would have been calculated $s=4055$ times, which is avoided due to the structure provided by $\equiv$.
\end{example}

\section{Application to reliability theory}\label{reliability}

In this section, we illustrate the use of constructive covers in reliability theory.

In the setting of \cite{barlow}, let us consider a \emph{multi-state coherent system} with $n\in\N$ components forming a set $\NN$. Without loss of generality, let us write $\NN=\cb{1,\ldots,n}$. Each component $p\in\NN$ has a state variable $z_p$ taking values in the set $\mathcal{S}\coloneqq\cb{0,1,\ldots,s}$, where $s\in\N$ is fixed for all components. Then, the states of all components can be expressed as a vector $z=(z_1,\ldots,z_n)\in\mathcal{S}^n$. In classical reliability theory, the typical systems are binary, that is, one takes $\mathcal{S}=\cb{0,1}$. In such systems, the state $1$ corresponds to the ``functioning" state and $0$ corresponds to the ``failure" state. Hence, the general case with $s+1$ states can be used to model varying levels of perfection for how well the components operate.

The structure of the system is encoded by a function $\phi\colon\mathcal{S}^n\to\mathcal{S}$ such that $\phi(z)=\phi(z_1,\ldots,z_n)$ gives the state of the overall system when component $p$ is at state $z_p$ for each $p\in\NN$. We call $\phi$ the \emph{structure function} of the system. For instance, a \emph{parallel} system is defined as a system whose structure function is given by $\phi(z)=\max\cb{z_1,\ldots,z_n}$ for each $z\in\mathcal{S}^n$. Similarly, a \emph{series} system has the structure function $\phi(z)=\min\cb{z_1,\ldots,z_n}$ for each $z\in\mathcal{S}^n$. In general, a \emph{coherent system} is a system which is built as a nested structure of parallel and series systems. More precisely, let $k\in\cb{1,\ldots,n}$ and consider distinct sets $P_1,\ldots,P_k\subseteq\NN$ satisfying the following properties: the sets are not strict subsets of each other, and $\bigcup_{i=1}^k P_i=\NN$. Then, the structure function of a coherent system described with these sets is defined by
\begin{equation}\label{phipath}
\phi(z)=\max_{i\in\cb{1,\ldots,k}}\min_{p\in P_i}z_p
\end{equation}
for each $z\in\mathcal{S}^n$. The sets $P_1,\ldots,P_k$ are called the \emph{minimal path sets} of the system.

Note that the above requirements for the minimal path sets of the system are precisely the conditions in the definition of constructive ordered $k$-cover (Definition~\ref{kcover}) except that the order of the sets is not important. Hence, in view of Remark~\ref{unordered}, every constructive unordered $k$-cover of $\NN$ corresponds to a system design with $k$ minimal path sets and $n$ components, and vice versa. By Corollary~\ref{ell0forcounting}, the number of such system designs is given by
\[
\frac{|C(\NN,k)|}{k!}=\sum_{\ell=\ell_0}^{(2^k-2)\wedge n}\frac{\ell!}{k!}\tilde{s}(n,\ell) |F(k,\ell)|.
\]

Given a coherent system with minimal path sets $P_1,\ldots,P_k$, by following a certain procedure, one can construct the so-called \emph{minimal cut sets} $C_1,\ldots,C_r\subseteq\NN$ with $r\in\cb{1,\ldots,n}$ satisfying the following properties: the sets are not strict subsets of each other, $\bigcup_{j=1}^r C_j=\NN$ and each minimal cut set has a nonempty intersection with each minimal path set. It can be shown that \cite[Proposition~1.1]{barlow} the structure function of the system can also be computed by the formula
\begin{equation}\label{phicut}
\phi(z)=\min_{j\in\cb{1,\ldots,r}}\max_{p\in C_j}z_p
\end{equation}
for each $z\in\mathcal{S}^n$.

\begin{landscape}
\begin{table}[b]
	\centering
\begin{adjustbox}{width=1.4\textwidth,totalheight=0.92\textheight-2\baselineskip}
\begin{tabular}{|c|*{28}{p{1cm}|}}\hline 
			\rowcolor{lightgray}
			\diagbox{$(a,w,s)$}{$\ell$} &$4$  & $ 5$ &  $ 6$ & $ 7$ & $8$& $9$ & $ 10$ & $11$ & $12$ & $13$
			&$14$  & $ 15$ &  $ 16$ & $ 17$ & $18$& $19$ & $ 20$ & $21$ & $22$ & $23$
			&$24$  & $ 25$ &  $ 26$ & $ 27$ & $28$& $29$ & $ 30$ \\ 	
			\hline 
			\cellcolor{lightgray} $(1,0,1)$& \cellcolor{lightred} $0$& \cellcolor{lightred} $38$& \cellcolor{lightred} $615$& \cellcolor{lightred} $2245$& \cellcolor{lightred} $3900$& \cellcolor{lightred} $4055$& \cellcolor{lightred} $2798$& \cellcolor{lightred} $1345$& \cellcolor{lightred} $455$& \cellcolor{lightred} $105$& \cellcolor{lightred} $15$& \cellcolor{lightred} $1$& \cellcolor{lightred} $0$& \cellcolor{lightred} $0$& \cellcolor{lightred} $0$& \cellcolor{lightred} $0$& \cellcolor{lightred} $0$& \cellcolor{lightred} $0$& \cellcolor{lightred} $0$& \cellcolor{lightred} $0$& \cellcolor{lightred} $0$& \cellcolor{lightred} $0$& \cellcolor{lightred} $0$& \cellcolor{lightred} $0$& \cellcolor{lightred} $0$& \cellcolor{lightred} $0$& \cellcolor{lightred} $0$ \\ \hline \cellcolor{lightgray} $(2,0,5)$& \cellcolor{lightgreen} $0$& \cellcolor{lightgreen} $11$& \cellcolor{lightgreen} $193$& \cellcolor{lightgreen} $1105$& \cellcolor{lightgreen} $2889$& \cellcolor{lightgreen} $4350$& \cellcolor{lightgreen} $4235$& \cellcolor{lightgreen} $2838$& \cellcolor{lightgreen} $1349$& \cellcolor{lightgreen} $455$& \cellcolor{lightgreen} $105$& \cellcolor{lightgreen} $15$& \cellcolor{lightgreen} $1$& \cellcolor{lightgreen} $0$& \cellcolor{lightgreen} $0$& \cellcolor{lightgreen} $0$& \cellcolor{lightgreen} $0$& \cellcolor{lightgreen} $0$& \cellcolor{lightgreen} $0$& \cellcolor{lightgreen} $0$& \cellcolor{lightgreen} $0$& \cellcolor{lightgreen} $0$& \cellcolor{lightgreen} $0$& \cellcolor{lightgreen} $0$& \cellcolor{lightgreen} $0$& \cellcolor{lightgreen} $0$& \cellcolor{lightgreen} $0$ \\
\hline
			\end{tabular}
\end{adjustbox}
\caption{$|F_{z^{T,a,w}}(5, \ell)|$ values for $\ell_0(k)\leq \ell\leq 2^k-2$ in Example \ref{k=5ex}}
\end{table}
\end{landscape}		

\begin{landscape}			
\begin{table}[b]
\centering
\begin{adjustbox}{width=1.4\textwidth,totalheight=0.92\textheight-2\baselineskip}
\begin{tabular}{|c|*{28}{p{1cm}|}}\hline 
\rowcolor{lightgray}
\diagbox{$(a,w,s)$}{$\ell$} &$4$  & $ 5$ &  $ 6$ & $ 7$ & $8$& $9$ & $ 10$ & $11$ & $12$ & $13$
						&$14$  & $ 15$ &  $ 16$ & $ 17$ & $18$& $19$ & $ 20$ & $21$ & $22$ & $23$
						&$24$  & $ 25$ &  $ 26$ & $ 27$ & $28$& $29$ & $ 30$ \\ 	
						\hline 
			 \cellcolor{lightgray} $(20,1,90)$& \cellcolor{lightgreen} $0$& \cellcolor{lightgreen} $0$& \cellcolor{lightgreen} $0$& \cellcolor{lightgreen} $16$& \cellcolor{lightgreen} $160$& \cellcolor{lightgreen} $740$& \cellcolor{lightgreen} $2100$& \cellcolor{lightgreen} $4088$& \cellcolor{lightgreen} $5776$& \cellcolor{lightgreen} $6105$& \cellcolor{lightgreen} $4895$& \cellcolor{lightgreen} $2981$& \cellcolor{lightgreen} $1363$& \cellcolor{lightgreen} $455$& \cellcolor{lightgreen} $105$& \cellcolor{lightgreen} $15$& \cellcolor{lightgreen} $1$& \cellcolor{lightgreen} $0$& \cellcolor{lightgreen} $0$& \cellcolor{lightgreen} $0$& \cellcolor{lightgreen} $0$& \cellcolor{lightgreen} $0$& \cellcolor{lightgreen} $0$& \cellcolor{lightgreen} $0$& \cellcolor{lightgreen} $0$& \cellcolor{lightgreen} $0$& \cellcolor{lightgreen} $0$ \\ \hline \cellcolor{lightgray} $(20,2,60)$& \cellcolor{lightgreen} $0$& \cellcolor{lightgreen} $0$& \cellcolor{lightgreen} $0$& \cellcolor{lightgreen} $0$& \cellcolor{lightgreen} $16$& \cellcolor{lightgreen} $160$& \cellcolor{lightgreen} $740$& \cellcolor{lightgreen} $2100$& \cellcolor{lightgreen} $4088$& \cellcolor{lightgreen} $5776$& \cellcolor{lightgreen} $6105$& \cellcolor{lightgreen} $4895$& \cellcolor{lightgreen} $2981$& \cellcolor{lightgreen} $1363$& \cellcolor{lightgreen} $455$& \cellcolor{lightgreen} $105$& \cellcolor{lightgreen} $15$& \cellcolor{lightgreen} $1$& \cellcolor{lightgreen} $0$& \cellcolor{lightgreen} $0$& \cellcolor{lightgreen} $0$& \cellcolor{lightgreen} $0$& \cellcolor{lightgreen} $0$& \cellcolor{lightgreen} $0$& \cellcolor{lightgreen} $0$& \cellcolor{lightgreen} $0$& \cellcolor{lightgreen} $0$ \\
\hline
	\end{tabular}
\end{adjustbox}
\ContinuedFloat
\caption{$|F_{z^{T,a,w}}(5, \ell)|$ values for $\ell_0(k)\leq \ell\leq 2^k-2$ in Example \ref{k=5ex}  (continued)}\label{k=5table1}
\end{table}
\end{landscape}		

 \begin{table}
	\scriptsize
	\centering
	\rule{0pt}{5ex}
	\resizebox{\textwidth}{!}{%
		\begin{tabular}{|c|*{16}{p{1cm}|}}
			\hline
			\rowcolor{lightgray}
			\rule{0pt}{2ex}
			\diagbox{\small{$(a,w)$}}{\small{$T$}} &$t^1$ & $t^2$ &$t^3$ & $t^4$ &$t^5$ & $t^6$ &$t^7$ & $t^8$ &$t^9$ & $t^{10}$ &$t^{11}$ & $t^{12}$ &$t^{13}$ & $t^{14}$ &$t^{15}$ \\ 
			\hline
			$\cellcolor{lightgray}(1,0)$&\cellcolor{lightred}$0$ &\cellcolor{lightred}$0$ &\cellcolor{lightred}$0$ &\cellcolor{lightred}$0$ &\cellcolor{lightred}$0$ &\cellcolor{lightred}$0$ &\cellcolor{lightred}$0$ &\cellcolor{lightred}$0$ &\cellcolor{lightred}$0$ &\cellcolor{lightred}$0$ &\cellcolor{lightred}$0$ &\cellcolor{lightred}$0$ &\cellcolor{lightred}$0$ &\cellcolor{lightred}$0$ &\cellcolor{lightred}$0$ \\ \hline$\cellcolor{lightgray}(2,0)$&\cellcolor{lightgreen}$1$ &\cellcolor{lightred}$0$ &\cellcolor{lightred}$0$ &\cellcolor{lightred}$0$ &\cellcolor{lightred}$0$ &\cellcolor{lightred}$0$ &\cellcolor{lightred}$0$ &\cellcolor{lightred}$0$ &\cellcolor{lightred}$0$ &\cellcolor{lightred}$0$ &\cellcolor{lightred}$0$ &\cellcolor{lightred}$0$ &\cellcolor{lightred}$0$ &\cellcolor{lightred}$0$ &\cellcolor{lightred}$0$ \\ \hline$\cellcolor{lightgray}(3,0)$&\cellcolor{lightred}$0$ &\cellcolor{lightred}$0$ &\cellcolor{lightred}$0$ &\cellcolor{lightred}$0$ &\cellcolor{lightred}$0$ &\cellcolor{lightgreen}$1$ &\cellcolor{lightred}$0$ &\cellcolor{lightred}$0$ &\cellcolor{lightred}$0$ &\cellcolor{lightred}$0$ &\cellcolor{lightred}$0$ &\cellcolor{lightred}$0$ &\cellcolor{lightred}$0$ &\cellcolor{lightred}$0$ &\cellcolor{lightred}$0$ \\ \hline$\cellcolor{lightgray}(4,0)$&\cellcolor{lightgreen}$1$ &\cellcolor{lightgreen}$1$ &\cellcolor{lightred}$0$ &\cellcolor{lightred}$0$ &\cellcolor{lightred}$0$ &\cellcolor{lightred}$0$ &\cellcolor{lightred}$0$ &\cellcolor{lightred}$0$ &\cellcolor{lightred}$0$ &\cellcolor{lightred}$0$ &\cellcolor{lightred}$0$ &\cellcolor{lightred}$0$ &\cellcolor{lightred}$0$ &\cellcolor{lightred}$0$ &\cellcolor{lightred}$0$ \\
\hline

	\end{tabular}}
	\caption{Definitions of the representatives $z^{T,a,w}$ in Example \ref{k=5ex}}\label{k=5table2}
\end{table}

\begin{table}[htb]
	\begin{tabular}{|>{\columncolor{lightgray}}P{1.2cm}|P{2.4cm}|P{2.4cm}|P{2.4cm}|P{2.4cm}|P{2.4cm}|}
		\hline
		\rowcolor{lightgray}
		\diagbox{$T$}{$\A$} &$A_1$& $A_2$&  $A_3$& $A_4$&$A_5$ \\
		\hline
		$t^1$&\cellcolor{lightgreen}$c$ &\cellcolor{lightred}$e$ &\cellcolor{lightred}$e$ &\cellcolor{lightred}$e$ &\cellcolor{lightred}$e$ \\ \hline$t^2$&\cellcolor{lightred}$e$ &\cellcolor{lightgreen}$c$ &\cellcolor{lightred}$e$ &\cellcolor{lightred}$e$ &\cellcolor{lightred}$e$ \\ \hline$t^3$&\cellcolor{lightred}$e$ &\cellcolor{lightred}$e$ &\cellcolor{lightgreen}$c$ &\cellcolor{lightred}$e$ &\cellcolor{lightred}$e$ \\ \hline$t^4$&\cellcolor{lightred}$e$ &\cellcolor{lightred}$e$ &\cellcolor{lightred}$e$ &\cellcolor{lightgreen}$c$ &\cellcolor{lightred}$e$ \\ \hline$t^5$&\cellcolor{lightred}$e$ &\cellcolor{lightred}$e$ &\cellcolor{lightred}$e$ &\cellcolor{lightred}$e$ &\cellcolor{lightgreen}$c$ \\ \hline$t^6$&\cellcolor{lightgreen}$c$ &\cellcolor{lightgreen}$c$ &\cellcolor{lightred}$e$ &\cellcolor{lightred}$e$ &\cellcolor{lightred}$e$ \\ \hline$t^7$&\cellcolor{lightgreen}$c$ &\cellcolor{lightred}$e$ &\cellcolor{lightgreen}$c$ &\cellcolor{lightred}$e$ &\cellcolor{lightred}$e$ \\
\hline
	\end{tabular}
	\caption{Definition of the set $T$ in Example \ref{k=5ex}}\label{k=5table3}
\end{table}

Let us comment on the alternative ways of calculating the state of the system given in \eqref{phipath} and \eqref{phicut}. We observe that the system is in the failure state if there is at least one minimal cut set in which every component is in the failure state, and the system is in a functioning state if there is at least one minimal path set in which every component is in a functioning state. Indeed, let $z\in\mathcal{S}^n$ and $\bar{s}\in\mathcal{S}$. Thanks to \eqref{phipath}, we have $\phi(z)\leq \bar{s}$ if and only if, for every minimal path set $P_i$, there is at least one component $p\in P_i$ with $z_p \leq \bar{s}$. Conversely, thanks to \eqref{phicut}, we have $\phi(z)>\bar{s}$ if and only if, for every minimal cut set $C_j$, there is at least one component $p\in C_j$ with $z_p>\bar{s}$.

As mentioned above, the minimal cut sets corresponding to a collection of minimal path sets can be constructed by a certain procedure. To be able to describe this procedure, we review the notion of minimality next. Let $\tilde{\mathscr{C}}=\{\tilde{C}_1,\ldots,\tilde{C}_{\tilde{r}}\}$ be a collection of distinct subsets of $\NN$ with $\tilde{r}\in\cb{1,\ldots,n}$. For $i\in\cb{1,\ldots,\tilde{r}}$, the set $\tilde{C}_i$ is called a \emph{dominated element} of $\tilde{\mathscr{C}}$ if there is $i^\prime\in\cb{1,\ldots,\tilde{r}}\sm\cb{i}$ such that $\tilde{C}_{i^\prime}\subseteq \tilde{C}_i$; in this case, $\tilde{C}_i$ is also said to be \emph{dominated by} $\tilde{C}_{i^\prime}$. $\tilde{C}_k$ is called a \emph{minimal element} of $\tilde{\mathscr{C}}$ if it is not a dominated element. Since $\tilde{\mathscr{C}}$ is a finite collection, a minimal element of it always exists. Moreover, it is easy to observe that every dominated element is dominated by at least one minimal element. It is easy to see that $\tilde{\mathscr{C}}$ is an constructive unordered $\tilde{r}$-cover of $\NN$ if and only if every set in $\tilde{\mathscr{C}}$ is a minimal element of $\tilde{\mathscr{C}}$ and $\bigcup_{i=1}^{\tilde{r}} \tilde{C}_i=\NN$.

Let $\mathscr{P}=\cb{P_1,\ldots,P_k}$ be a constructive unordered $k$-cover of $\NN$. A set $C\subseteq \cb{1,\ldots,n}$ is called an \emph{intersector} of $\mathscr{P}$ if there exists a surjective function $f\colon\cb{1,\ldots,k}\to C$ such that $f(i)\in P_i$ for every $i\in\cb{1,\ldots,k}.$ In other words, an intersector chooses one component from each set in $\mathscr{P}$ but the same component can be chosen from multiple sets. Consequently, for an intersector $C$ of $\mathscr{P}$, it holds $C\cap P_i\neq \emptyset$ for each $i\in\cb{1,\ldots,k}$ but the cardinality of such an intersection may exceed one, in general. Moreover, for every component $p\in\NN$, there exists an intersector $C$ of $\mathscr{P}$ such that $p\in C$. Let us denote by $\mathscr{C}_0$ the collection of all intersectors of $\mathscr{P}$ and by $\mathscr{C}$ the collection of all minimal elements of $\mathscr{C}_0$. Let us write
\begin{equation}\label{minint}
\mathscr{C}=\cb{C_1,\ldots,C_r},
\end{equation}
where $r=|\mathscr{C}|$.

\begin{proposition}
	The collection $\mathscr{C}$ given in \eqref{minint} is a constructive unordered $r$-cover of $\NN$.
\end{proposition}

\begin{proof}
	We first note that, by construction, every set in $\mathscr{C}$ is a minimal element of $\mathscr{C}_0$, hence it is a minimal element of $\mathscr{C}$ as well. It remains to show that $\bigcup_{j=1}^r C_j =\NN$. The $\subseteq$ part is obvious since $C_1,\ldots,C_r$ are subsets of $\NN$. To prove the $\supseteq$ part, let us fix $p\in\NN$. Let us denote by $\mathscr{C}_0(p)$ the collection of all intersectors of $\mathscr{P}$ containing $p$ as an element and by $\mathscr{C}(p)$ the collection of all minimal elements of $\mathscr{C}_0(p)$. Let us write $\mathscr{C}(p)=\cb{K_1,\ldots,K_q}$ with $q\in\N$ as a collection of distinct sets. To conclude the proof, it suffices to show that $\mathscr{C}(p)\subseteq\mathscr{C}$. To that end, we let $\a\in\cb{1,\ldots,q}$ and show that $K_\a\in\mathscr{C}$. To get a contradiction, suppose that $K_\a\notin\mathscr{C}$. Note that $\mathscr{C}(p)\subseteq\mathscr{C}_0$. Hence, the supposition is equivalent to that, $K_\a$ is a dominated element of $\mathscr{C}_0$ so that there exists $j(\a)\in\cb{1,\ldots,r}$ such that $C_{j(\a)}\subseteq K_\a$. Moreover, such $C_{j(\a)}$ does not contain $p$ as an element as otherwise $C_{j(\a)}$ would be a set in $\mathscr{C}_0(p)$, which would contradict the minimality of $K_j$ in $\mathscr{C}_0(p)$. Since $C_{\a(j)}$ is an intersector of $\mathscr{P}$, there exists a surjective function $f\colon\cb{1,\ldots,k}\to C_{j(\a)}$ such that $f(i)\in P_i$ for every $i\in\cb{1,\ldots,k}$. Define a function $g\colon\cb{1,\ldots,k}\to \NN$ by letting $g(i)=p$ if $p\in P_{i} $ and $g(i)=f(i)$ if $p\notin P_i$. Since $\mathscr{P}$ is constructive, there exists $i^\prime$ such that $p\in P_{i^\prime}$. Consequently, the image $g(\cb{1,\ldots,k})$ of $g$ is an intersector of $\mathscr{P}$ containing $p$ as an element, $C_{j(\a)}\sm g(\cb{1,\ldots,m})\neq \emptyset$, and $g(\cb{1,\ldots,k})\subseteq C_{j(\a)}\cup\cb{p}\subseteq K_j$. Hence, $g(\cb{1,\ldots,k})\in\mathscr{C}_0(p)$ and the minimality of $K_j$ in $\mathscr{C}_0(p)$ implies $C_{j(\a)}\subseteq K_j=g(\cb{1,\ldots,k})$, which contradicts $C_{j(\a)}\sm g(\cb{1,\ldots,k})\neq \emptyset$. Therefore, $K_j\in\mathscr{C}$.
\end{proof}

As an alternative construction, one can start with a constructive unordered $r$-cover $\{C_1,\ldots,C_r\}$ of $\NN$ to be used as the collection of the minimal cut sets of the system, that is, the structure function is defined by \eqref{phicut}. Then, the corresponding minimal path sets can be constructed by the above procedure, which guarantees that each minimal path set has a nonempty intersection with each minimal cut set. The number of all system designs with $r$ minimal cut sets is given by
\[
\frac{|C(\NN,r)|}{r!}=\sum_{\ell=\ell_0(k)}^{(2^r-2)\wedge n}\frac{\ell!}{r!}\tilde{s}(n,\ell) |F(r,\ell)|.
\]
We finish this section by providing two exact calculations of the above quantity based on the earlier calculations.

\begin{example}
	Using Example\ref{k=4ex}, we calculate the number of all system designs with $n=7$ components and $r=4$ minimal cut sets as
	\begin{align*}
	\frac{|C(\NN,r)|}{r!}&=\frac{1}{24}\sum_{\ell=4}^{7}\ell!\tilde{s}(7,\ell) |F(4,\ell)|\\
	&=\frac{1}{24}\of{4!\tilde{s}(7,4)|F(4,4)|+5!\tilde{s}(7,5)|F(4,5)|+6!\tilde{s}(7,6)|F(4,6)|+7!\tilde{s}(7,7)|F(4,7)|}\\
	&=\frac{1}{24}\of{4!\cdot 1050\cdot 25+5!\cdot 266\cdot 304 + 6!\cdot 28\cdot 1165+7!\cdot 1\cdot 2188}\\
	&=1,868,650.
	\end{align*}
	\end{example}

\begin{example}
	Using Example\ref{k=5ex}, we calculate the number of all system designs with $n=9$ components and $r=5$ minimal cut sets as
	\begin{align*}
	\frac{|C(\NN,r)|}{r!}&=\frac{1}{120}\sum_{\ell=4}^{9}\ell!\tilde{s}(9,\ell) |F(5,\ell)|\\
	&=\frac{1}{120}(4!\tilde{s}(9,4)|F(5,4)|+5!\tilde{s}(9,5)|F(5,5)|+6!\tilde{s}(9,6)|F(5,6)|+7!\tilde{s}(9,7)|F(5,7)|\\
	&\quad \quad\quad +8!\tilde{s}(9,8)|F(5,8)|+9!\tilde{s}(9,9)|F(5,9)|)\\
	&=\frac{1}{120}(4!\cdot42525\cdot 30 + 5!\cdot 22827\cdot 2026+6!\cdot 5880\cdot 41430+7!\cdot 750 \cdot 376350\\
	&\quad \quad\quad +8!\cdot 45 \cdot 2003655+9!\cdot 1 \cdot 7286000)\\
	&=65,691,305,652.
	\end{align*}
\end{example}

\section{Conclusion}\label{conc}

In this paper, we consider the problem of counting the number of ways that one can cover a finite set by a given number of subsets with the additional requirement that every pair of these subsets has a nonempty set difference. It turns out that this seemingly simple problem requires a deep enumeration argument which results in two auxiliary problems: calculating ISNs and counting certain labelings of the disjoint regions that are induced by a cover. While the calculation of ISNs can be handled by a simple recursive relation, we solve the labeling problem by a more sophisticated method that exploits certain symmetries available in the set of labelings and uses no-good cuts from optimization literature. As the numerical examples illustrate, even for small values of $k$, the number of subsets in the cover, one has to calculate very large cardinalities due to exponential growth. The enhancement of this method for larger values of $k$ as well as the asymptotic analysis of the overall problem are subjects of future research.

\section{Appendix: Proofs of the results in Section~\ref{prelim}}\label{app}

\begin{proof}[Proof of Proposition~\ref{workhorse}]
	Using binomial expansion, we obtain
	\begin{align*}
	\tilde{s}(n,\ell)&=\frac{1}{\ell!}\sum_{j=0}^\ell (-1)^{\ell-j}\binom{\ell}{j}(j+1)^n=\frac{1}{\ell!}\sum_{j=0}^\ell (-1)^{\ell-j}\binom{\ell}{j}\sum_{i=0}^n\binom{n}{i}j^i\\
	&=\frac{1}{\ell!}\sum_{i=0}^n\binom{n}{i}\sum_{j=0}^\ell (-1)^{\ell-j}\binom{\ell}{j}j^i=\frac{1}{\ell!}\sum_{i=1}^n\binom{n}{i}\sum_{j=0}^\ell (-1)^{\ell-j}\binom{\ell}{j}j^i=\sum_{i=1}^n\binom{n}{i}s(i,\ell).
	\end{align*}
	Here, to get the penultimate equality, we use the fact that, for $i=0$,
	\[
	\sum_{j=0}^\ell (-1)^{\ell-j}\binom{\ell}{j}j^i=\sum_{j=0}^\ell (-1)^{\ell-j}(+1)^j\binom{\ell}{j}=0.
	\]
	Hence, \eqref{stilderes} follows. If $n<\ell$, then \eqref{Stirlingzero} implies
	\[
	\tilde{s}(n,\ell)=\sum_{i=1}^n\binom{n}{i}s(i,\ell)=\sum_{i=1}^n\binom{n}{i}0=0.
	\]
	On the other hand, if $n\geq \ell$, then we have $s(i,\ell)=0$ for each $i\in\cb{1,\ldots,\ell-1}$ so that
	\[
	\tilde{s}(n,\ell)=\sum_{i=\ell}^n\binom{n}{i}s(i,\ell)=\sum_{i=0}^{n-\ell}\binom{n}{i}s(n-i,\ell).
	\]
	Hence, \eqref{stildeinterpret} follows. In particular, $\tilde{s}(n,n)=\binom{n}{0} s(n,n)=1$. Finally, by \eqref{stilderes} together with (i) and (iii), we have
	\[
	\tilde{s}(n,1)=\sum_{i=1}^n \binom{n}{i}s(i,1)=\sum_{i=1}^n \binom{n}{i}=2^n-1.
	\] 
\end{proof}

\begin{proof}[Proof of Proposition~\ref{stilderec}]
	Let $n\in\N\sm\cb{1}$ and $\ell\in\cb{2,\ldots,n}$. By elementary calculations, we obtain
	\begin{align*}
	&(\ell+1)\tilde{s}(n,\ell)+\tilde{s}(n,\ell-1)\\
	&=\frac{\ell+1}{\ell!}\sum_{j=0}^\ell (-1)^{\ell-j}\binom{\ell}{j}(j+1)^n+\frac{1}{(\ell-1)!}\sum_{j=0}^{\ell-1} (-1)^{\ell-1-j}\binom{\ell-1}{j}(j+1)^n\\
	&=\frac{\ell+1}{\ell!}\sum_{j=0}^{\ell-1} (-1)^{\ell-j}\binom{\ell}{j}(j+1)^n+\frac{\ell+1}{\ell!}(\ell+1)^n+\frac{1}{(\ell-1)!}\sum_{j=0}^{\ell-1} (-1)^{\ell-1-j}\binom{\ell-1}{j}(j+1)^n\\
	&=\frac{\ell+1}{\ell}\frac{1}{(\ell-1)!}\sum_{j=0}^{\ell-1} (-1)^{\ell-j}\binom{\ell}{j}(j+1)^n+\frac{\ell+1}{\ell!}(\ell+1)^n+\frac{(-1)^{-1}}{(\ell-1)!}\sum_{j=0}^{\ell-1} (-1)^{\ell-j}\binom{\ell-1}{j}(j+1)^n\\
	&=\frac{1}{(\ell-1)!}\sum_{j=0}^{\ell-1} (-1)^{\ell-j}(j+1)^n\of{\frac{\ell+1}{\ell}\binom{\ell}{j}-\binom{\ell-1}{j}}+\frac{\ell+1}{\ell!}(\ell+1)^n\\
	&=\frac{1}{(\ell-1)!}\sum_{j=0}^{\ell-1} (-1)^{\ell-j}(j+1)^n\binom{\ell}{j}\of{\frac{\ell+1}{\ell}-\frac{\ell-j}{\ell}}+\frac{\ell+1}{\ell!}(\ell+1)^n\\
	&=\frac{1}{\ell!}\sum_{j=0}^{\ell-1} (-1)^{\ell-j}\binom{\ell}{j}(j+1)^{n+1}+\frac{1}{\ell!}(\ell+1)^{n+1}=\frac{1}{\ell!}\sum_{j=0}^{\ell} (-1)^{\ell-j}\binom{\ell}{j}(j+1)^{n+1}=\tilde{s}(n+1,\ell),
	\end{align*}
	as desired. The boundary conditions are given by (iii), (iv) of Proposition~\ref{workhorse}.
\end{proof}

\bibliographystyle{named}

\end{document}